\documentclass[a4paper]{article}

\usepackage{amsfonts, latexsym, amssymb, amsthm, amsmath, mathrsfs, esint, mathtools, enumitem, stmaryrd}
\usepackage{authblk}
\usepackage[colorlinks=true]{hyperref}

\newcommand{\R}{\mathbb{R}}
\newcommand{\N}{\mathbb{N}}
\newcommand{\D}{\mathcal{D}}
\newcommand{\M}{\mathbf{M}}
\newcommand{\Le}{\mathcal{L}}
\newcommand{\Ha}{\mathcal{H}}
\newcommand{\loc}{\mathrm{loc}}

\let\div\relax\DeclareMathOperator{\div}{div}

\DeclareMathOperator{\dist}{dist}
\DeclareMathOperator{\supp}{supp}
\DeclareMathOperator*{\esssup}{ess \, sup}
\newcommand{\blank}{{\mkern 2mu\cdot\mkern 2mu}}
\newcommand{\rep}[2]{\llbracket #1 \rrbracket_{#2}}
\newcommand{\locrep}[2]{\llparenthesis #1 \rrparenthesis_{#2}}

\newcommand{\dd}[2]{\frac{\partial #1}{\partial #2}}
\newcommand{\set}[2]{\left\{ #1 \colon #2 \right\}}

\newcommand{\restr}{\mathchoice
{\kern2pt\mbox{\vrule width 0.08ex height1.5ex depth0ex\kern-0.08ex\vrule width 1.5ex height.08ex depth0ex}\kern2pt}
{\kern2pt\mbox{\vrule width 0.08ex height1.5ex depth0ex\kern-0.08ex\vrule width 1.5ex height.08ex depth0ex}\kern2pt}
{\kern1.5pt\mbox{\vrule width 0.06ex height1.1ex depth0ex\kern-0.06ex\vrule width 1.1ex height.06ex depth0ex}\kern1.5pt}
{\kern1pt\mbox{\vrule width 0.04ex height0.75ex depth0ex\kern-0.04ex\vrule width 0.75ex height.04ex depth0ex}\kern1pt}
}

\newtheorem{theorem}{Theorem}
\newtheorem{lemma}[theorem]{Lemma}
\newtheorem{proposition}[theorem]{Proposition}
\newtheorem{definition}[theorem]{Definition}
\newtheorem{corollary}[theorem]{Corollary}
\newtheorem*{problem}{Problem}
\theoremstyle{remark}
\newtheorem*{notation}{Notation}{}

\newtheorem*{remark}{Remark}{}

\newenvironment{subproof}[1]{\medskip\noindent \textit{#1} \,}{}

\setlist[enumerate, 1]{label = (\roman*), ref = (\roman*)}
\setlist[enumerate, 2]{label = (\alph*), ref = (\alph*)}

\begin{document}

\title{Minimisers of supremal functionals and mass-minimising $1$-currents}

\author{Nikos Katzourakis \\ \small Department of Mathematics and Statistics, University of Reading, Whiteknights, Pepper Lane, Reading, RG6 6AX, UK. E-mail: n.katzourakis@reading.ac.uk
\and Roger Moser \\ \small Department of Mathematical Sciences,
University of Bath, Bath BA2 7AY, UK.
E-mail: r.moser@bath.ac.uk}

\maketitle

\begin{abstract}
We study vector-valued functions that minimise the $L^\infty$-norm of their derivatives
for prescribed boundary data. We construct a vector-valued, mass minimising $1$-current
(i.e., a generalised geodesic) in the domain such that all solutions of the problem
coincide on its support. Furthermore, this current can be interpreted as a streamline of the solutions.
The construction relies on a $p$-harmonic approximation. In the case of
scalar-valued functions, it is closely related to a construction of Evans and Yu
\cite{Evans-Yu:05}. We therefore obtain an extension of their theory.
\end{abstract}

\section{Introduction}

For $n \in \N$, let $\Omega \subseteq \R^n$ be a bounded domain with smooth boundary. Given $N \in \N$,
we study functions $u \colon \Omega \to \R^N$ that minimise the functional
\[
E_\infty(u) = \esssup_\Omega |Du|
\]
for prescribed boundary data, where $|\blank|$ denotes the Frobenius norm of an $(N \times n)$-matrix.

Variational problems of this sort go back to the pioneering work of
Aronsson \cite{Aronsson:65, Aronsson:66, Aronsson:67, Aronsson:68}. The
scalar case $N = 1$ is now quite well understood. For the above functional $E_\infty$, it gives
rise to the Aronsson equation, which has a well-developed theory in the framework of viscosity solutions.
It has unique solutions for given boundary data, which correspond not just to minimisers
of the functional, but to so-called absolute minimisers \cite{Jensen:93}. These are characterised by
the condition that they minimise $\esssup_{\Omega'} |D u|$ in suitable subsets
$\Omega' \subseteq \Omega$ for their own boundary data. There is a regularity theory for
solutions of the equation as well \cite{Savin:05, Evans-Savin:08, Evans-Smart:11.2, Koch-Zhang-Zhou:19, Dong-Peng-Zhang-Zhou:24}.

Much less is known for the vector-valued case $N \ge 2$, despite some work by the first author
\cite{Katzourakis:12, Katzourakis:13, Katzourakis:14.2, Katzourakis:14.1, Katzourakis:15.2, Katzourakis:17.1}.
(This problem should not be confused with the problem of vector-valued optimal Lipschitz extensions \cite{Sheffield-Smart:12},
which amounts to replacing the Frobenius norm with the operator norm. For $n = 1$ or $N = 1$, the two problems
are equivalent, but in general, they are not.)
The vector-valued counterpart to the Aronsson equation is easy to write down, but much more
difficult to make sense of. Above all, there is no meaningful interpretation of viscosity
solutions. It is not known if absolute minimisers exist in general or if they are unique. We therefore
take a somewhat different point of view in this paper.
Rather than trying to determine a distinguished solution of the problem, we ask what all minimisers
of $E_\infty$ for given boundary data have in common.

The natural space for the functional $E_\infty$ is the Sobolev space $W^{1, \infty}(\Omega; \R^N)$.
We prescribe boundary data given in terms of a fixed Lipschitz continuous function
$u_0 \colon \R^n \to \R^N$,
and then we are interested in the subset $u_0 + W_0^{1, \infty}(\Omega; \R^N)$. We define
\[
e_\infty = \inf_{u_0 + W_0^{1, \infty}(\Omega; \R^N)} E_\infty.
\]
We are thus interested in the following problem.

\begin{problem}
Study all maps $u_\infty \in u_0 + W^{1, \infty}(\Omega; \R^N)$ such that $E_\infty(u_\infty) = e_\infty$.
\end{problem}

Since $u_0$ is Lipschitz continuous, it has a tangential derivative
at almost every boundary point with respect to the $(n - 1)$-dimensional Hausdorff measure.
We write $D' u_0$ for the tangential derivative of $u_0$, and we define
\[
e_\infty' = \esssup_{\partial \Omega} |D' u_0|
\]
(where the essential supremum is taken with respect to the $(n - 1)$-dimensional Hausdorff measure on $\partial \Omega$).
The case where $e_\infty' < e_\infty$ is particularly interesting for the reasons explained below.

One of the key ingredients in our analysis is a measure derived as a limit from $p$-harmonic approximations.
This tool essentially goes back to an idea of Evans \cite{Evans:03.2} and has been studied by Evans and Yu \cite{Evans-Yu:05}
for the case $N = 1$, albeit in a different form.
We show here that there is an interesting geometric structure behind this limit measure. Moreover,
it is useful for the vector-valued problem, too, even though this requires an approach rather different
from the one taken by Evans and Yu.

Roughly speaking, we will find some generalised length-minimising `curves' in the domain, along
which all solutions of the problem must coincide. In order to formulate a precise statement, however,
we need some tools from geometric measure theory, above all the notion of currents.

Since we will work only with $1$-currents (and their boundaries, given by $0$-currents),
readers not familiar with geometric measure theory may think of vector-valued  distributions instead,
or indeed of vector-valued measures in most cases. This will, however, obscure some of the
geometric content of our results.

\begin{definition}[Current, mass, boundary]
For $j = 0, \dotsc, n$, let $\D^j(\R^n)$ denote the space of smooth, compactly supported $j$-forms in $\R^n$, endowed with the topology
induced by locally uniform convergence of all derivatives. A \emph{$j$-current} in $\R^n$
is an element of its dual space $\D_j(\R^n) = (\D^j(\R^n))^*$. Given a $j$-current $T \in \D_j(\R^n)$, its \emph{mass}
is\footnote{It is more common to define the mass in terms of the co-mass norm on the space of $j$-covectors, but this
definition appears, e.g., in a book by Simon \cite{Simon:83}. In the context of this paper, the distinction is inconsequential, because
we will work only with $1$-currents and their boundaries.}
\[
\M(T) = \sup\set{T(\omega)}{\omega \in \D^j(\R^n) \text{ with } \sup_{\R^n} |\omega| \le 1}.
\]
If $j \ge 1$, its \emph{boundary} is the $(j - 1)$-current $\partial T$ such that $\partial T(\sigma) = T(d\sigma)$
for every $\sigma \in \D^{j - 1}(\R^n)$.
\end{definition}

We are interested above all in $1$-currents supported on $\overline{\Omega}$
and their boundaries (which are $0$-currents and can be regarded as
distributions). These can be interpreted as generalised oriented curves in $\overline{\Omega}$.
Indeed, given an oriented $C^1$-curve $\Gamma \subset \overline{\Omega}$ of finite length,
we can define a corresponding $1$-current $T$ by integration of $1$-forms over $\Gamma$, i.e.,
\begin{equation} \label{eq:current-curve}
T(\omega) = \int_\Gamma \omega.
\end{equation}
In this situation, the mass $\M(T)$ will be the length of $\Gamma$.
More generally, if $\mu$ is a Radon measure on $\overline{\Omega}$ and $\tau$ is a vector field that is integrable
with respect to $\mu$, then we can define a $1$-current $T = [\mu, \tau]$ by the formula
\begin{equation} \label{eq:current-measure}
T(\omega) = \int_{\overline{\Omega}} \omega(\tau) \, d\mu.
\end{equation}
In this case,
\[
\M(T) = \int_{\overline{\Omega}} |\tau| \, d\mu
\]
and
\[
\partial T(\sigma) = \int_{\overline{\Omega}} d\sigma(\tau) \, d\mu
\]
for a $0$-form $\sigma$ on $\R^n$.
For example, suppose that we have the above oriented curve $\Gamma$. Write $\Ha^1$ for the one-dimensional Hausdorff measure
and define $\mu = \Ha^1 \restr \Gamma$. Furthermore, let $\tau$ denote the unit tangent vector field along $\Gamma$
consistent with the orientation. Then the formulas \eqref{eq:current-curve} and \eqref{eq:current-measure} give
rise to the same current.

We will also need to study $N$-tuples $T = (T_1, \dotsc, T_N)$ of $1$-currents, which may also be regarded as
vector-valued currents. The boundary $\partial T$ is then taken component-wise. We use the following extension of the mass.

\begin{definition}[Joint mass, normal]
Let $T = (T_1, \dotsc, T_N) \in (D_j(\R^n))^N$ be an $N$-tuple of $j$-currents in $\R^n$. Then the \emph{joint mass}
of $T$ is
\[
\M(T) = \sup\set{\sum_{k = 1}^N T_k(\omega_k)}{\omega_1, \dotsc, \omega_N \in \D^j(\R^n) \text{ with } \sup_{\R^n} \sum_{k = 1}^N |\omega_k|^2 \le 1}.
\]
For $j \ge 1$, we say that $T$ is \emph{normal} if $\M(T) < \infty$ and $\M(\partial T) < \infty$.
\end{definition}

Given $T = (T_1, \dotsc, T_N) \in (D_1(\R^n))^N$ with $\M(T) < \infty$, we can always find a Radon measure $\|T\|$ on
$\R^n$ and $\|T\|$-measurable vector fields $\vec{T}_1, \dotsc, \vec{T}_N$ such that
$\sum_{k = 1}^n |\vec{T}_k|^2 = 1$
almost everywhere and $T_k = [\|T\|, \vec{T}_k]$ for $k = 1, \dotsc, N$. Furthermore, this representation is unique
(up to identification of vector fields that agree $\|T\|$-almost everywhere).
We then write $\vec{T} = (\vec{T}_1, \dotsc, \vec{T}_N)$.

Similarly, if $\M(\partial T) < \infty$, then $\partial T$ is represented by
an $\R^N$-valued Radon measure. For any continuous function
$u \colon \R^n \to \R^N$ with compact support, we may then write
\[
\partial T(u) = \sum_{k = 1}^N \partial T_k(u_k).
\]
If $\supp \partial T \subseteq \overline{\Omega}$, then
this also makes sense for any $u \in W^{1, \infty}(\Omega; \R^N)$, including the minimisers of $E_\infty$.

For a minimiser $u_\infty \in u_0 + W_0^{1, \infty}(\Omega; \R^N)$ of $E_\infty$,
and for $T$ as above, we want to be able to make
statements about the `behaviour of $D u_\infty$ along $\|T\|$'. Since the support of
$\|T\|$ can be a null set with respect to the Lebesgue measure on $\Omega$, and since $D u_\infty$ is only
well-defined up to null sets, this requires some explanation.

\begin{definition}[Local $L^2$-representative] \label{def:restriction-local}
Let $\mu$ be a (non-negative) Radon measure on $\Omega$ and let $f \in L_\loc^1(\Omega)$. We say that $g \in L_\loc^2(\mu)$ is the
\emph{local $L^2$-representative} of $f$ with respect to $\mu$, and we write
\[
g = \locrep{f}{\mu},
\]
if for any compact set $K \Subset \Omega$ and
for any $\eta \in C_0^\infty(\R^n)$ with $\int_{\R^n} \eta(x) \, dx = 1$, the functions
$\eta_\epsilon(x) = \epsilon^{-n} \eta(x/\epsilon)$ satisfy
\[
\lim_{\epsilon \searrow 0} \int_K |g - \eta_\epsilon * f|^2 \, d\mu = 0.
\]
The concept is defined similarly for vector-valued functions.
\end{definition}

We require some more notation.
For $x \in \R^n$ and $r > 0$, let $B_r(x)$ denote the open ball in $\R^n$ with centre $x$ and radius $r$.
Given a measurable function $f \colon \Omega \to [0, \infty)$, we define $f^\star \colon \overline{\Omega} \to [0, \infty]$ by
\[
f^\star(x) = \lim_{r \searrow 0} \esssup_{B_r(x) \cap \Omega} f
\]
for $x \in \overline{\Omega}$.

We can now formulate our first main result.

\begin{theorem} \label{thm:main}
There exists an $N$-tuple of $1$-currents $T = (T_1, \dotsc, T_N)$ of finite joint mass
with the following properties.
\begin{enumerate}
\item \label{itm:boundary} $\supp T \subseteq \overline{\Omega}$ and $\supp \partial T \subseteq \partial \Omega$.
\item \label{itm:minimiser} For any minimiser $u = (u_1, \dotsc, u_N) \in u_0 + W_0^{1, \infty}(\Omega; \R^N)$ of $E_\infty$,
\begin{enumerate}
\item $|Du|^\star(x) = e_\infty$ for $\|T\|$-almost every $x \in \Omega$;
\item $\locrep{Du}{\|T\|} = e_\infty \vec{T}$.
\end{enumerate}
\item \label{itm:unique} Any two minimisers of $E_\infty$ in $u_0 + W_0^{1, \infty}(\Omega; \R^N)$ coincide on $\supp T$.
\item \label{itm:non-trivial} If $e_\infty' < e_\infty$, then the following holds true.
\begin{enumerate}
\item \label{itm:normal} $T$ is normal.
\item \label{itm:trivial-on-boundary} $\|T\|(\partial \Omega) = 0$ and $\|T\|(\Omega) > 0$.
\item \label{itm:optimal} $\partial T(u_0) = e_\infty \M(T)$.
\item \label{itm:geodesic} Let $S$ be a normal $N$-tuple of $1$-currents
such that $\supp S \subseteq \overline{\Omega}$ and $\supp \partial S \subseteq \partial \Omega$.
If $\partial S = \partial T$, then $\M(S) \ge \M(T)$.
\end{enumerate}
\end{enumerate}
\end{theorem}

The theorem may appear technical, but there is a geometric interpretation. Statement \ref{itm:boundary}
says that we have $N$ generalised curves $T_1, \dotsc, T_N$ in $\overline{\Omega}$ with no boundary in the interior
of $\Omega$.
Along these generalised curves, according to \ref{itm:minimiser}, any solution $u$ of our variational
problem will have a derivative of norm $e_\infty$ (in a certain sense), with $Du_k$ tangent to $T_k$.
We therefore have a generalisation of what is called a \emph{streamline} in
the theory of $\infty$-harmonic functions \cite{Aronsson:67}.
(For the vector-valued case, similar observations have been made by the first author for sufficiently
smooth solutions \cite{Katzourakis:12}.)
Moreover, we have uniqueness of minimisers on the support of $T$ by \ref{itm:unique}.

If $e_\infty' < e_\infty$, then we can make additional statements on the structure of $T$ according to
\ref{itm:non-trivial}. It follows in particular that $T$ is non-trivial in the interior of $\Omega$
by \ref{itm:trivial-on-boundary}. Furthermore, statement \ref{itm:geodesic} says that 
$T$ minimises the joint mass (generalised length) among all competitors with the same boundary in this case.
This means that we can think of \emph{generalised
length-minimising geodesics} rather than just generalised curves.

If $e_\infty' = e_\infty$, then the theorem may appear vacuous, because it does not rule out that $T$
is supported completely on the boundary or even that $T = 0$.
The proof, however, reveals some more information, which is omitted here for the sake
of brevity. In our construction, the $1$-current $T$ will arise from a limit of
$p$-harmonic functions $u_p \colon \Omega \to \R^N$, or more precisely, from
renormalisations of the currents $[|D u_p|^{p - 2} \mathcal{L}^n, Du_p]$, where $\mathcal{L}^n$ is
the Lebesgue measure on $\R^n$. It is possible that $\supp T \subseteq \partial \Omega$,
and then the theorem is indeed vacuous, but it can be interesting even if $e_\infty' = e_\infty$.
In any case, $\|T\|$ coincides with the measure studied by Evans and Yu \cite{Evans-Yu:05} for $N = 1$.
They proved a condition that amounts to \ref{itm:boundary} and some
additional properties related to \ref{itm:minimiser} and \ref{itm:unique}.
The analysis of this paper can therefore also be regarded as an extension of their results.
Statement \ref{itm:non-trivial} is completely new even if $N = 1$.

In the case $e_\infty' = e_\infty$, we still have a local variant of \ref{itm:non-trivial}.\ref{itm:geodesic}.
Since its formulation is even more technical, however, we postpone it until Section \ref{sct:proof}
(see Theorem \ref{thm:local-geodesic}).
We can further derive some information about the structure of $T$, which we will do in Section \ref{sct:structure}
(see Theorem \ref{thm:structure}).
We employ standard results from geometric measure theory here, even though we consider non-standard
objects. This result can be thought of as a consequence of the the mass minimising property
in statement \ref{itm:non-trivial}.\ref{itm:geodesic}, but our construction allows a shortcut.

There is a simple observation that exemplifies some of the behaviour described
in Theorem \ref{thm:main}. Suppose that there are two boundary points $x, y \in \partial \Omega$ such that
the line segment $L$ between $x$ and $y$ is contained in $\overline{\Omega}$.
Suppose further that $|u_0(x) - u_0(y)| \ge e_\infty |x - y|$. Then, if $u$ is a minimiser of $E_\infty$ in $u_0 + W_0^{1, \infty}(\Omega; \R^N)$,
we can see quite easily that the restriction of $u$ to $L$ is affine with a derivative of norm $e_\infty$.
In general, such a pair of points need not exist, but the $N$-tuple of currents $T$ from Theorem \ref{thm:main}
has properties similar to the line segment $L$.

In the case $N = 1$, it was in fact proved by Aronsson \cite[Theorem 2]{Aronsson:67} that the set where all
minimisers coincide is characterised by line segments as above. (Other related results also exist
\cite{Brizzi-DePascale:23, Brizzi:22}.) This is not true for $N > 1$.
Consider, e.g., the case $n = N$. If $u_0(x) = x$ for all $x \in \Omega$, then $u_0$ is $p$-harmonic for all
$p < \infty$, and our analysis shows that it is also the unique minimiser of $E_\infty$ in $u_0 + W_0^{1, \infty}(\Omega; \R^n)$.
But for any $x, y \in \partial \Omega$, we clearly have the identity $|u_0(x) - u_0(y)| = |x - y|$, while
$e_\infty = \sqrt{n}$ in this case.

We have interpreted $T$ from Theorem \ref{thm:main} as a generalised curve, but in general it need not actually
be $1$-dimensional. For example, suppose that $u_0$ is an affine function. Then, as in the preceding
example, we conclude that every $p$-harmonic function for these boundary data coincides
with $u_0$, and an examination of our construction reveals that $\|T\|$ is a normalised version of
the Lebesgue measure in $\Omega$ and $\vec{T}$ is constant. We can then interpret $T$ as the collective
representation of many line segments in $\Omega$.

If $e_\infty' = e_\infty$, then general statements on the behaviour of minimisers of $E_\infty$ in
the interior of $\Omega$ cannot be expected.
In this case, the minimum value of $E_\infty$ is in fact dictated by the local behaviour near a single boundary point.

\begin{proposition} \label{prp:boundary-point}
If $e_\infty = e_\infty'$, then there exists $x \in \partial \Omega$ such that $|Du|^\star(x) \ge e_\infty$
for any $u \in u_0 + W_0^{1, \infty}(\Omega; \R^N)$.
\end{proposition}

In this situation, we may have a minimiser $u \in u_0 + W_0^{1, \infty}(\Omega; \R^N)$
such that $\|Du\|_{L^\infty(K)} < e_\infty$ for any compact set $K \subset \Omega$.
Then for any $\phi \in W_0^{1, \infty}(\Omega; \R^N)$ with
compact support, there exists $\delta > 0$ such that $u_\infty + t\phi$ is
also a minimiser of $E_\infty$ for all $t \in (-\delta, \delta)$.
Thus the class of minimisers is simply too large to expect any general statements. (For example,
let $n = 2$ and $N = 1$. Let $r = \sqrt{2}/(\sqrt{2} + 1)$ and consider the domain
$\Omega = B_r(r, r) \subseteq \R^2$. Suppose that $u_0(x_1, x_2) = x_1^{4/3} - x_2^{4/3}$. This happens
to be a minimiser of $E_\infty$ for its boundary values \cite{Aronsson:84}. We find that $e_\infty = e_\infty' = 4\sqrt{2}/3$, which
is the value of $|D u_0|$ at $(1, 1)$ but nowhere else in $\overline{\Omega}$.)
If $e_\infty' < e_\infty$, on the other hand, Theorem \ref{thm:main} rules out such behaviour.

Variational problems in $L^\infty$ have long been studied predominantly with methods involving comparison
arguments and viscosity solutions of the corresponding partial differential equations. Even so, measure theoretic
arguments have also made an appearance in the literature in various contexts
\cite{Evans-Yu:05, Champion-DePascale-Jimenez:09, Bungert-Korolev:22, Katzourakis:22}.
This also includes a paper with a more geometric point of view by Daskalopoulos and Uhlenbeck \cite{Daskalopoulos-Uhlenbeck:22},
which is motivated by work of Thurston \cite{Thurston:98}. Their paper studies `$\infty$-harmonic maps' from
a hyperbolic manifold to the circle and shows (among other things) that the locus of maximum
stretch is a geodesic lamination and is contained in the support of a certain $(n - 1)$-current.
(It should be noted here that $(n - 1)$-currents can be identified with $1$-currents via the Hodge star operator.)
There is thus some overlap with the above theory, but the paper of Daskalopoulos and Uhlenbeck is
restricted to one-dimensional targets and its scope is different.
Further related results can be found in a work by Backus \cite{Backus:24}.

\begin{notation}
The following notation will be convenient. Given $r > 0$, we define
$\Omega_r = \set{x \in \Omega}{\dist(x, \partial \Omega) < r}$.
Given two matrices $A, B \in \R^{N \times n}$, we write $A : B$ for their Frobenius inner product.
\end{notation}

\section{Currents and what they say about $E_\infty$} \label{sct:currents}

In this section, we give some more information on the relationship between $N$-tuples of $1$-currents and
the functions minimising $E_\infty$ in $u_0 + W_0^{1, \infty}(\Omega; \R^N)$.

We first formulate a global version of Definition \ref{def:restriction-local}. We cannot simply use the
convolution with a mollifying kernel any more, because this would cause problems at the boundary.
Instead, we use the following.

\begin{definition}[Regular mollifier] \label{def:mollifier}
A family of linear operators $\mathcal{M}_\epsilon \colon L^\infty(\Omega) \to C^\infty(\overline{\Omega})$, for $\epsilon > 0$, is called
a \emph{regular mollifier} if
\begin{enumerate}
\item \label{itm:pointwise-bound}
there exists $\theta \colon (0, \infty) \to [0, \infty)$ such that
$\lim_{\epsilon \searrow 0} \theta(\epsilon) = 0$ and
\[
|(\mathcal{M}_\epsilon f)(x)| \le \|f\|_{L^\infty(\Omega \cap B_\epsilon(x))} + \theta(\epsilon)
\]
for all $x \in \overline{\Omega}$;
\item if $f \in C^0(\overline{\Omega})$, then $\mathcal{M}_\epsilon f \to f$ uniformly as $\epsilon \searrow 0$; and
\item \label{itm:commutator}
if $f \in W^{1, \infty}(\Omega)$, then
\[
\left|\mathcal{M}_\epsilon \dd{f}{x_j} - \dd{}{x_j} \mathcal{M}_\epsilon f\right| \to 0
\]
in $C^0(\overline{\Omega})$ as $\epsilon \searrow 0$ for $j = 1, \dotsc, n$.
\end{enumerate}
\end{definition}

Regular mollifiers can be constructed by modifying the usual convolution with a mollifying
kernel. For example, a suitable approach is used in a recent paper by the authors
\cite[Proof of Lemma 7]{Katzourakis-Moser:23}. Some tools for a different construction are
discussed in a paper by the first author \cite[Section 5]{Katzourakis:22}.

\begin{definition}[$L^2$-representative] \label{def:restriction-global}
Let $\mu$ be a (non-negative)
Radon measure on $\overline{\Omega}$ and let $f \in L^\infty(\Omega)$. We say that $g \in L^2(\mu)$ is the
\emph{$L^2$-representative} of $f$ with respect to $\mu$, and we write
\[
g = \rep{f}{\mu},
\]
if for every regular mollifier $\mathcal{M}_\epsilon$,
\[
\lim_{\epsilon \searrow 0} \int_{\overline{\Omega}} |g - \mathcal{M}_\epsilon f|^2 \, d\mu = 0.
\]
The concept is defined similarly for vector-valued functions.
\end{definition}

We have the following connection between currents and the functional $E_\infty$.

\begin{proposition} \label{prp:optimal-current-global}
Suppose that $T = (T_1, \dotsc, T_N)$ is a normal $N$-tuple of $1$-currents in $\R^n$ with $\supp T \subseteq \overline{\Omega}$.
Let $u \in W^{1, \infty}(\Omega; \R^N)$. Then $\partial T(u) \le E_\infty(u) \M(T)$, with equality if, and only if,
$\rep{D u}{\|T\|} = E_\infty(u) \vec{T}$.
\end{proposition}

\begin{proof}
If $E_\infty(u) = 0$, then the statement is trivial. We therefore assume that $E_\infty(u) > 0$.

Consider a regular mollifier $(\mathcal{M}_\epsilon)_{\epsilon > 0}$ and set $u_\epsilon = \mathcal{M}_\epsilon u$. We define
$e = E_\infty(u)$. Then we compute
\begin{equation} \label{eq:current-global}
\begin{split}
\int_{\overline{\Omega}} |e \vec{T} - D u_\epsilon|^2 \, d\|T\| & = \int_{\overline{\Omega}} (e^2 + |D u_\epsilon|^2) \, d\|T\| - 2e \int_{\overline{\Omega}} \vec{T} : D u_\epsilon \, d\|T\| \\
& = \int_{\overline{\Omega}} (e^2 + |D u_\epsilon|^2) \, d\|T\| - 2e \partial T(u_\epsilon).
\end{split}
\end{equation}
By conditions \ref{itm:pointwise-bound} and \ref{itm:commutator} in Definition \ref{def:mollifier}, we know that
\[
\limsup_{\epsilon \searrow 0} \|D u_\epsilon\|_{C^0(\overline{\Omega})} \le \|D u\|_{L^\infty(\Omega)} \le e.
\]
It follows that
\[
\limsup_{\epsilon \searrow 0} \int_{\overline{\Omega}} (e^2 + |D u_\epsilon|^2) \, d\|T\| \le 2e^2 \M(T).
\]
On the other hand, we have the uniform convergence $u_\epsilon \to u$. As $T$ is normal, it follows that
$\partial T(u) = \lim_{\epsilon \searrow 0} \partial T(u_\epsilon)$. Taking the limsup in
\eqref{eq:current-global}, we find that
\[
\partial T(u) + \frac{1}{2e} \limsup_{\epsilon \searrow 0} \int_{\overline{\Omega}} |e \vec{T} - D u_\epsilon|^2 \, d\|T\| \le e \M(T).
\]
Hence $\partial T(u) \le e\M(T)$. If we have equality, then it further follows that $\rep{D u}{\|T\|} = e\vec{T}$.

Conversely, suppose that $\rep{Du}{\|T\|} = e\vec{T}$. Then $\mathcal{M}_\epsilon Du \to e\vec{T}$ in
$L^2(\|T\|)$. By property \ref{itm:commutator} in Definition \ref{def:mollifier}, this means that
$D u_\epsilon \to e\vec{T}$ in $L^2(\|T\|)$ as well. Hence
\[
\partial T(u) = \lim_{\epsilon \searrow 0} \partial T(u_\epsilon) = \lim_{\epsilon \searrow 0} \int_{\overline{\Omega}} D u_\epsilon : \vec{T} \, d\|T\| = e \int_{\overline{\Omega}} |\vec{T}|^2 \, d\|T\| = e\M(T),
\]
as claimed.
\end{proof}

\begin{remark}
These arguments also show that if $\mathcal{M}_\epsilon Du \to E_\infty(u) \vec{T}$
in $L^2(\|T\|)$ for some \emph{specific} regular mollifier $\mathcal{M}_\epsilon$, then
$\rep{Du}{\|T\|} = E_\infty(u) \vec{T}$.
\end{remark}

\begin{corollary} \label{cor:minimal-current-global}
Suppose that $T = (T_1, \dotsc, T_N)$ is a normal $N$-tuple of $1$-currents in $\R^n$ with $\supp T \subseteq \overline{\Omega}$
and $\supp \partial T \subseteq \partial \Omega$. Then $\partial T(u_0) \le e_\infty \M(T)$. If equality holds and $e_\infty > 0$,
then $\M(S) \ge \M(T)$ for any normal $N$-tuple of $1$-currents $S$ such that $\supp S \subseteq \overline{\Omega}$ and
$\partial S = \partial T$.
\end{corollary}

\begin{proof}
In order to prove the first statement, it suffices to choose a minimiser $u_\infty \in u_0 + W_0^{1, \infty}(\Omega; \R^N)$
of $E_\infty$ and apply Proposition \ref{prp:optimal-current-global} to $u_\infty$. If we have equality and if
$e_\infty > 0$, we apply the first statement to $S$. We conclude that
$\M(T) = e_\infty^{-1} \partial T(u_0) = e_\infty^{-1} \partial S(u_0) \le \M(S)$.
\end{proof}

The following are local versions of the above statements.

\begin{corollary} \label{cor:optimal-current-local}
Suppose that $T = (T_1, \dotsc, T_N)$ is an $N$-tuple of $1$-currents in $\R^n$ with locally finite mass
and with $\supp \partial T \cap \Omega = \emptyset$.
Let $u \in W^{1, \infty}(\Omega; \R^N)$. Then
\[
-T(ud\chi) \le E_\infty(u) \int_\Omega \chi \, d\|T\|
\]
for any $\chi \in C_0^\infty(\Omega)$ with $\chi \ge 0$. Equality holds for every such function $\chi$ if,
and only if, $\locrep{Du}{\|T\|} = E_\infty(u) \vec{T}$.
\end{corollary}

\begin{proof}
Fix $\chi \in C_0^\infty(\Omega)$ with $\chi \ge 0$.
Consider $S \in (\D_1(\R^n))^N$ defined by $S_k(\omega) = T_k(\chi \omega)$ for $\omega \in (\D^1(\R^n))$
and for $k = 1, \dotsc, N$.
Then $S$ is normal with $\partial S(u) = -T(ud\chi)$ and
\[
\M(S) = \int_\Omega \chi \, d\|T\|.
\]
The first statement therefore follows immediately from Proposition \ref{prp:optimal-current-global}.

Let $\eta \in C_0^\infty(\R^n)$ with $\int_{\R^n} \eta \, dx = 1$. Let $\eta_\epsilon(x) = \epsilon^{-n} \eta(x/\epsilon)$.
Then we can construct a regular
mollifier $\mathcal{M}_\epsilon$ such that $\mathcal{M}_\epsilon f = \eta_\epsilon * f$ in $\supp \chi$ for all
$f \in L^\infty(\Omega)$. Therefore, the second statement follows from Proposition \ref{prp:optimal-current-global}
and the remark after its proof.
\end{proof}

\begin{corollary} \label{cor:minimising-current-local}
Let $T = (T_1, \dotsc, T_N)$ be an $N$-tuple of $1$-currents in $\R^n$ with locally finite mass
and with $\supp \partial T \cap \Omega = \emptyset$.
Let $u \in u_0 + W_0^{1, \infty}(\Omega; \R^N)$. Suppose that
\[
-T(u d\chi) \ge e_\infty \int_\Omega \chi \, d\|T\|
\]
for every $\chi \in C_0^\infty(\Omega)$ with $\chi \ge 0$. Then the following holds true.
\begin{enumerate}
\item \label{itm:lower-bound}
For any open set $\Omega' \subseteq \Omega$ with $\supp T \cap \Omega' \neq \emptyset$,
\[
\esssup_{\Omega'} |D u| \ge e_\infty.
\]
\item \label{itm:locally-length-minimising}
Suppose that $e_\infty > 0$. Let $K \subset \Omega$ be a compact set.
Let $S \in (\D_1(\R^n))^N$ with locally finite mass
and with $\supp \partial S \cap \Omega = \emptyset$. If $S(\omega) = T(\omega)$ for every
$\omega \in (\D^1(\R^n))^N$ with $\omega = 0$ in $K$, then
\[
\|T\|(K) \le \|S\|(K).
\]
\end{enumerate}
\end{corollary}

\begin{proof}
We may assume that $e_\infty > 0$, because \ref{itm:lower-bound} is trivial otherwise and \ref{itm:locally-length-minimising}
excludes $e_\infty = 0$.

Given an open set $\Omega' \subseteq \Omega$ with $\supp T \cap \Omega' \neq \emptyset$,
choose $\chi \in C_0^\infty(\Omega')$ with $\chi \ge 0$ and $\int_{\Omega'} \chi \, d\|T\| > 0$.
Choose an open set $\Omega'' \subseteq \Omega'$ with smooth boundary such that $\supp \chi \subset \Omega''$.
Applying Corollary \ref{cor:optimal-current-local} in $\Omega''$, we find that
\[
-T(u d\chi) \le \esssup_{\Omega''} |D u| \int_{\Omega''} \chi \, d\|T\|.
\]
But by the assumptions, we also have
\[
-T(u d\chi) \ge e_\infty \int_{\Omega''} \chi \, d\|T\|.
\]
Hence $\esssup_{\Omega'} |D u_\infty| \ge e_\infty$.

If $S$ and $K$ are as in the second statement, then we choose
$\chi \in C_0^\infty(\Omega)$ such that $\chi \equiv 1$ in $K$ and $0 \le \chi \le 1$ everywhere. Then
\[
\int_\Omega \chi \, d\|T\| \le - \frac{1}{e_\infty} T(u d\chi) = - \frac{1}{e_\infty} S(u d\chi) \le \int_\Omega \chi \, d\|S\|
\]
by Corollary \ref{cor:optimal-current-local}.
Since $\|T\|(K)$ and $\|S\|(K)$ are the limits of such integrals when $\chi$ approaches the
characteristic function of $K$, the claim follows.
\end{proof}

If $T$ has a mass minimising property as in Corollary \ref{cor:minimal-current-global} or Corollary
\ref{cor:minimising-current-local}, then we may think of it as a generalised length-minimising geodesic.
There is an Euler-Lagrange equation for this variational problem, which amounts to the condition that
\begin{equation} \label{eq:geodesic}
\sum_{k = 1}^N \int_\Omega \vec{T}_k \cdot D_{\vec{T}_k} \psi \, d\|T\| = 0
\end{equation}
for all $\psi \in C_0^\infty(\Omega; \R^n)$, where $D_{\vec{T}_k}$ denotes the directional derivative
in the direction of $\vec{T}_k$.
This equation can be derived
with the same arguments as for the more conventional (scalar-valued) mass minimising currents.
(These tools are even more common in the theory of varifolds \cite{Allard:72}, but this is just another side of the same coin for
this purpose.) We do not need to go into the details here, because we will obtain
the equation in a different way. Instead, we formulate a consequence.

\begin{proposition} \label{prp:pointwise}
Suppose that $T$ is an $N$-tuple of $1$-currents such that \eqref{eq:geodesic} holds true for all
$\psi \in C_0^\infty(\Omega; \R^n)$. Let $w \in W_0^{1, \infty}(\Omega; \R^N)$.
If $\locrep{Dw}{\|T\|} = 0$, then $w = 0$ on $\supp T$.
\end{proposition}

\begin{proof}
Let $r > 0$. Let $\chi \in C_0^\infty(\Omega)$ with $\chi \equiv 1$ in $\Omega \setminus \Omega_r$ and
such that $|D\chi| \le 2r$. Let $\eta \in C_0^\infty(\R^n)$ with $\int_{\R^n} \eta \, dx = 1$ and set
$\eta_\epsilon(x) = \epsilon^{-n} \eta(x/\epsilon)$ for $\epsilon > 0$. Define $w_\epsilon = \eta_\epsilon * w$,
assuming that $w$ is extended (arbitrarily) to $\R^n$.

Set
\[
\psi_\epsilon(x) = |w_\epsilon(x)|^2 \chi(x) x
\]
for $x \in \Omega$. Then we may test \eqref{eq:geodesic} with $\psi_\epsilon$.
This gives
\[
\begin{split}
0 & = 2\sum_{k = 1}^N \int_\Omega x \cdot \vec{T}_k \, (w_\epsilon \otimes \vec{T}_k) : Dw_\epsilon \chi \, d\|T\| \\
& \quad + \sum_{k = 1}^N \int_\Omega x \cdot \vec{T}_k \, |w_\epsilon|^2 D_{\vec{T}_k}\chi  \, d\|T\| + \int_\Omega |w_\epsilon|^2 \chi \, d\|T\|.
\end{split}
\]

Since $\locrep{Dw}{\|T\|} = 0$, we have the convergence
\[
\int_{\supp D\chi} |Dw_\epsilon|^2 \, d\|T\| \to 0
\]
as $\epsilon \searrow 0$. Clearly $w_\epsilon \to w$ locally uniformly. Hence
\[
0 = \sum_{k = 1}^N \int_\Omega x \cdot \vec{T}_k \, |w|^2 D_{\vec{T}_k}\chi \, d\|T\| + \int_\Omega |w|^2 \chi \, d\|T\|.
\]
If we let $r \searrow 0$, then the first term vanishes in the limit, because
\[
|w|^2 \le N E_\infty(w) r^2
\]
in $\Omega_r$. Hence
\[
0 = \int_\Omega |w|^2 \, d\|T\|.
\]
It follows that $w = 0$ almost everywhere with respect to $\|T\|$. By the continuity of $w$, we conclude that $w = 0$
on $\supp T$.
\end{proof}

\section{Measure-function pairs and $L^p$-approximation} \label{sct:measure-function}

Here we introduce another tool from geometric measure theory, due to Hutchinson \cite{Hutchinson:86},
that is convenient for our purpose. We only discuss the $L^2$-version of Hutchinson's
theory here, because this is all we need.

In the second part of this section, we will apply it to minimisers of the functionals
\begin{equation} \label{eq:E_p}
E_p(u) = \left(\fint_\Omega |Du|^p \, dx\right)^{1/p}.
\end{equation}
The limit
$p \to \infty$ will eventually produce not just a minimiser of $E_\infty$, but also the $1$-currents
from Theorem \ref{thm:main}.

\begin{definition}[Measure-function pair] \label{def:measure-function-pair}
A \emph{measure-function pair} over $\overline{\Omega}$ with values in $\R^{N \times n}$ is a pair $(\mu, F)$,
where $\mu$ is a Radon measure on $\overline{\Omega}$ and $F \in L^2(\mu; \R^{N \times n})$.
\end{definition}

Hutchinson further defines weak and strong convergence of measure-function pairs. His formulation of weak convergence
is somewhat weaker than the following, but the strong convergence is the same.

\begin{definition}[Weak and strong convergence]
For $\ell \in \N$, let $M_\ell = (\mu_\ell, F_\ell)$ be measure-function pairs over $\overline{\Omega}$ with
values in $\R^{N \times n}$. Let $M_\infty = (\mu_\infty, F_\infty)$ be another such measure-function pair.
\begin{enumerate}
\item
We have the \emph{weak convergence} $M_\ell \rightharpoonup M_\infty$ as $\ell \to \infty$ if
\[
\lim_{\ell \to \infty} \int_{\overline{\Omega}} (\eta + F_\ell : \phi) \, d\mu_\ell = \int_{\overline{\Omega}} (\eta + F_\infty : \phi) \, d\mu_\infty
\]
for any $\eta \in C^0(\overline{\Omega})$ and any $\phi \in C^0(\overline{\Omega}; \R^{N \times n})$, and at the same time,
\[
\limsup_{\ell \to \infty} \int_{\overline{\Omega}} (1 + |F_\ell|^2) \, d\mu_\ell < \infty.
\]
\item
We have the \emph{strong convergence} $M_\ell \to M_\infty$ as $\ell \to \infty$ if
\[
\lim_{a \to \infty} \int_{\set{x \in \overline{\Omega}}{|F_\ell(x)| \ge a}} |F_\ell|^2 \, d\mu_\ell = 0
\]
uniformly in $\ell$ and
\[
\lim_{\ell \to \infty} \int_{\overline{\Omega}} \Phi(x, F_\ell(x)) \, d\mu_\ell(x) = \int_{\overline{\Omega}} \Phi(x, F_\infty(x)) \, d\mu_\infty
\]
for all $\Phi \in C_0^0(\overline{\Omega} \times \R^{N \times n})$.
\end{enumerate}
\end{definition}

Some of the key statements in this theory are summarised in the following proposition, which amounts to a variant of
\cite[Theorem 4.4.2]{Hutchinson:86}. It generalises well-known
results for weak and strong $L^2$-convergence for a fixed measure $\mu$.

\begin{proposition} \label{prp:weak-compactness}
For $\ell \in \N$, suppose that $M_\ell = (\mu_\ell, F_\ell)$ are measure-function pairs
over $\overline{\Omega}$ with values in $\R^{N \times n}$.
Let $M_\infty = (\mu_\infty, F_\infty)$ be another such measure function-pair.
Let $\Phi \colon \overline{\Omega} \times \R^{N \times n} \to \R$ be a continuous function
and suppose that there exists a constant $C > 0$ such that
$|\Phi(x, z)| \le C(|z|^2 + 1)$ for all $x \in \overline{\Omega}$ and $z \in \R^{N \times n}$.
\begin{enumerate}
\item
If 
\[
\limsup_{\ell \to \infty} \int_{\overline{\Omega}}(|F_\ell|^2 + 1) \, d\mu_\ell < \infty,
\]
then there exists a subsequence $(M_{\ell_m})_{m \in \N}$ that converges weakly.
\item If $(M_\ell)_{\ell \in \N}$ converges weakly to $M_\infty$, then
\[
\|F_\infty\|_{L^2(\mu_\infty)} \le \liminf_{\ell \to \infty} \|F_\ell\|_{L^2(\mu_\ell)}.
\]
\item If $(M_\ell)_{\ell \in \N}$ converges weakly to $M_\infty$ and
\[
\|F_\infty\|_{L^2(\mu_\infty)} = \lim_{\ell \to \infty} \|F_\ell\|_{L^2(\mu_\ell)},
\]
then the convergence is strong.
\item
If $(M_\ell)_{\ell \in \N}$ converges strongly to $M_\infty$, then
\[
\int_{\overline{\Omega}} \Phi(x, F_\infty(x)) \, d\mu_\infty(x) = \lim_{\ell \to \infty} \int_{\overline{\Omega}} \Phi(x, F_\ell(x)) \, d\mu_\ell(x).
\]
\end{enumerate}
\end{proposition}

\begin{proof}
The first three statements are from \cite[Theorem 4.4.2]{Hutchinson:86}. The final statement is
a little stronger than what is stated in Hutchinson's paper, but can be proved with the same
arguments. (It is also quite easy to prove directly from the definition of strong convergence.)
\end{proof}

We will apply these concepts to measures and functions generated by $L^p$-approximations
of minimisers of $E_\infty$. For $2 \le p < \infty$, we therefore consider the functionals
$E_p$ defined in \eqref{eq:E_p}. Since they are strictly convex,
there exists a unique minimiser $u_p \in u_0 + W_0^{1, p}(\Omega; \R^n)$ of $E_p$
for each $p < \infty$,
which satisfies the Euler-Lagrange equation
\begin{equation} \label{eq:Euler-Lagrange-p}
\div(|D u_p|^{p - 2} D u_p) = 0
\end{equation}
weakly in $\Omega$. Furthermore, for any $\psi \in C_0^\infty(\Omega; \R^n)$, the condition
\[
0 = \left.\frac{d}{dt}\right|_{t = 0} E_p(u_p \circ(\operatorname{id}_\Omega + t\psi))
\]
gives rise to
\begin{equation} \label{eq:Pohozaev-p}
0 = \int_\Omega |D u_p|^{p - 2} \left(\sum_{i, j = 1}^n \dd{u_p}{x_i} \cdot \dd{u_p}{x_j} \, \dd{\psi_i}{x_j} - \frac{1}{p} |D u_p|^2 \div \psi\right) \, dx.
\end{equation}

Fix $v \in u_0 + W_0^{1, \infty}(\Omega; \R^N)$. If $2 \le q < p < \infty$, then we observe that
\begin{equation} \label{eq:monotonicity-1}
E_q(u_q) \le E_q(u_p) \le E_p(u_p) \le E_p(v) \le E_\infty(v)
\end{equation}
by H\"older's inequality and the definition of $u_p$. Hence the family $(u_p)_{p \in [2, \infty)}$ is
bounded in $W^{1, q}(\Omega; \R^N)$ for every $q < \infty$. Therefore, there exists a sequence
$p_\ell \to \infty$ such that $u_{p_\ell} \to u_\infty$ weakly in all of these spaces for some
limit $u_\infty \in u_0 + \bigcap_{q < \infty} W_0^{1, q}(\Omega; \R^N)$. We then estimate
\begin{equation} \label{eq:monotone-convergence}
E_\infty(u_\infty) = \lim_{q \to \infty} E_q(u_\infty) \le \lim_{q \to \infty} \liminf_{\ell \to \infty} E_q(u_{p_\ell}) \le \liminf_{\ell \to \infty} E_{p_\ell}(u_{p_\ell}) \le E_\infty(v),
\end{equation}
where we have used \eqref{eq:monotonicity-1} in the last two steps. Hence $u_\infty \in u_0+ W_0^{1, \infty}(\Omega; \R^N)$,
and $u$ is a minimiser of the functional $E_\infty$. In particular, it satisfies $E_\infty(u_\infty) = e_\infty$. We also define
$e_p = E_p(u_p)$. Then as in \eqref{eq:monotonicity-1}, we see that
\[
e_q \le e_p \le e_\infty
\]
whenever $q \le p \le \infty$. Moreover, the estimates in \eqref{eq:monotone-convergence} imply that
$e_p \to e_\infty$ monotonically.

We now define
\[
\mu_p = \frac{|D u_p|^{p - 2}}{e_p^{p - 2} \Le^n(\Omega)} \Le^n, 
\]
where $\Le^n$ denotes the Lebesgue measure on $\overline{\Omega}$.
These measures have also been studied by Evans and Yu \cite{Evans-Yu:05} in the case $N = 1$.
They should be considered not on their own, but in conjunction with the function $D u_p$.
Thus $M_p = (\mu_p, D u_p)$ naturally forms a measure-function pair over $\overline{\Omega}$.

We can make a few statements about $M_p$ immediately. The Euler-Lagrange equation \eqref{eq:Euler-Lagrange-p}
now becomes
\begin{equation} \label{eq:Euler-Lagrange}
\int_{\Omega} D u_p : D \phi \, d\mu_p = 0
\end{equation}
for every $\phi \in C_0^\infty(\Omega; \R^N)$. We also have
\begin{equation} \label{eq:Pohozaev}
\int_{\Omega} \left(\sum_{i, j = 1}^n \dd{u_p}{x_i} \cdot \dd{u_p}{x_j} \, \dd{\psi_i}{x_j} - \frac{1}{p} |D u_p|^2 \div \psi\right) \, d\mu_p = 0
\end{equation}
for every $\psi \in C_0^\infty(\Omega; \R^n)$ because of \eqref{eq:Pohozaev-p}.
Moreover, we observe that
\[
\int_{\overline{\Omega}} |D u_p|^2 \, d\mu_p = e_p^{2 - p} \fint_\Omega |D u_p|^p \, dx = e_p^2 \le e_\infty^2
\]
and
\[
\mu_p(\overline{\Omega}) = e_p^{2 - p} \fint_\Omega |D u_p|^{p - 2} \, dx \le e_p^{2 - p} (E_p(u_p))^{p - 2} = 1
\]
by H\"older's inequality. By Proposition \ref{prp:weak-compactness}, we may therefore replace $(p_\ell)_{\ell \in \N}$
with a subsequence such that $M_{p_\ell} \rightharpoonup M_\infty$ for some
measure-function pair $M_\infty = (\mu_\infty, F_\infty)$ over $\overline{\Omega}$. Then
\begin{equation} \label{eq:no-boundary}
\int_{\overline{\Omega}} F_\infty : D\phi \, d\mu_\infty = \lim_{\ell \to \infty} \int_{\overline{\Omega}} Du_{p_\ell} : D\phi \, d\mu_{p_\ell} = 0
\end{equation}
for every $\phi \in C_0^\infty(\Omega; \R^N)$ by \eqref{eq:Euler-Lagrange}.
We will see that \eqref{eq:Pohozaev} also gives a useful equation in the limit, but
we need to establish strong convergence first.

\section{Interesting boundary data}

In this section, we examine the case $e_\infty' < e_\infty$ in more detail. Recall that
\[
e_\infty = \inf_{u_0 + W_0^{1, \infty}(\Omega; \R^N)} E_\infty = E_\infty(u_\infty)
\]
and
\[
e_\infty' = \esssup_{\partial \Omega} |D' u_0|.
\]

If $e_\infty' < e_\infty$, then we have better control of the $p$-harmonic functions $u_p$
near the boundary $\partial \Omega$. Therefore, we can prove additional properties of the limiting
measure-function pair $M_\infty = (\mu_\infty, F_\infty)$.
The following estimates rely on identity \eqref{eq:Pohozaev}.
We use measure-function pairs over $\partial \Omega$ with values in $\R^N$ here, which are defined
similarly to Definition \ref{def:measure-function-pair}.

\begin{proposition} \label{prp:boundary-measures}
Suppose that $e_\infty' < e_\infty$. Then there exist measure-function pairs $(m_p, f_p)$ over $\partial \Omega$
with values in $\R^N$ such that
\begin{equation} \label{eq:boundary-measures}
\limsup_{p \to \infty} \int_{\partial \Omega} \left(1 + |f_p|^2\right) \, dm_p < \infty
\end{equation}
and
\begin{equation} \label{eq:boundary}
\int_\Omega Du_p : D\phi \, d\mu_p = \int_{\partial \Omega} f_p \cdot \phi \, dm_p
\end{equation}
for all $\phi \in C^\infty(\R^n; \R^N)$.
\end{proposition}

Comparing the last identity with the usual integration by parts formula for $p$-harmonic functions,
we may think of $m_p$ as the restriction of $|Du_p|^{p - 2}$ to the boundary (up to rescaling)
and of $f_p$ as a representation of $\nu \cdot Du_p$, where $\nu$ denotes the outer normal vector
on $\partial \Omega$. In general, however, we do not know if we have enough regularity of $u_p$
up to the boundary to write down the formulas in these terms.

\begin{proof}
We first replace $u_p$ by the solutions of a different boundary value problem, for which we have the
better regularity theory. We fix $p \in [2, \infty)$ at first. For $\epsilon > 0$, define the
function $H_{p, \epsilon} \colon \R \to \R$ by $H_{p, \epsilon}(t) = (t^2 + \epsilon^2)^{p/2}$ for
$t \in \R$. Consider the functional
\[
E_{p, \epsilon}(u) = \left(\fint_\Omega H_{p, \epsilon}(|Du|) \, dx\right)^{1/p}.
\]
Choose $c \in (e_\infty', e_\infty)$.
Let $u_{0, \epsilon} \colon \R^n \to \R^N$ be smooth functions such that
\begin{itemize}
\item $u_{0, \epsilon} \to u_0$ in $W^{1, p}(\Omega; \R^N)$ as $\epsilon \searrow 0$ and
\item $|D'u_{0, \epsilon}|^2 + \epsilon^2 \le c^2$ on $\partial \Omega$ when $\epsilon$ is sufficiently small.
\end{itemize}
Let $u_{p, \epsilon} \in u_{0, \epsilon} + W_0^{1, p}(\Omega; \R^N)$ be the
unique minimiser of $E_{p, \epsilon}$ in this space. Then we have the Euler-Lagrange
equation
\begin{equation} \label{eq:Euler-Lagrange-regularised}
\div\left(\left(|Du_{p, \epsilon}|^2 + \epsilon^2\right)^{p/2 - 1} Du_{p, \epsilon}\right) = 0
\end{equation}
in $\Omega$.
There are theories for both interior and boundary regularity for this sort of problem. In particular,
results of Uhlenbeck \cite{Uhlenbeck:77} show that $u_{p, \epsilon}$ is smooth in the interior of $\Omega$.
Results of Kristensen and
Mingione \cite[Theorem 1.1]{Kristensen-Mingione:10} show that $Du_{p, \epsilon}$ belongs to
$W^{1/p + s, p}(\Omega; \R^{N \times n})$ and its trace on $\partial \Omega$ belongs to
$W^{s, p}(\partial \Omega; \R^{N \times n})$ for some $s > 0$. This is enough to carry out the
following computations.

Equation \eqref{eq:Euler-Lagrange-regularised} implies that
\[
\dd{}{x_j} H_{p, \epsilon}(|Du_{p, \epsilon}|) = p\div\left(\left(|Du_{p, \epsilon}|^2 + \epsilon^2\right)^{p/2 - 1} \dd{u_{p, \epsilon}}{x_j} \cdot Du_{p, \epsilon}\right)
\]
in $\Omega$ for $j = 1, \dotsc, n$.
We choose a smooth vector field $\psi \colon \R^n \to \R^n$ such that $\psi = \nu$ on $\partial \Omega$.
Then
\begin{multline*}
\div\left(H_{p, \epsilon}(|Du_{p, \epsilon}|) \psi - p \left(|Du_{p, \epsilon}|^2 + \epsilon^2\right)^{\frac{p}{2} - 1} D_\psi u_{p, \epsilon} \cdot Du_{p, \epsilon}\right) \\
= H_{p, \epsilon}(|Du_{p, \epsilon}|) \div \psi - p \left(|Du_{p, \epsilon}|^2 + \epsilon^2\right)^{\frac{p}{2} - 1} \sum_{i, j = 1}^n \dd{u_{p, \epsilon}}{x_i} \cdot \dd{u_{p, \epsilon}}{x_j} \, \dd{\psi_i}{x_j}.
\end{multline*}
We write $\Ha^{n - 1}$ for the $(n - 1)$-dimensional Hausdorff measure. Then
\begin{multline*}
\int_{\partial \Omega} \left(|Du_{p, \epsilon}|^2 + \epsilon^2\right)^{\frac{p}{2} - 1} \left(|Du_{p, \epsilon}|^2 + \epsilon^2 - p |D_\nu u_{p, \epsilon}|^2\right) \, d\Ha^{n - 1} \\
\begin{aligned}
& = \int_\Omega H_{p, \epsilon}(|D u_{p, \epsilon}|) \div \phi \, dx \\
& \quad - p \sum_{i, j = 1}^n \int_\Omega \left(|Du_{p, \epsilon}|^2 + \epsilon^2\right)^{\frac{p}{2} - 1} \dd{u_{p, \epsilon}}{x_i} \cdot \dd{u_{p, \epsilon}}{x_j} \, \dd{\psi_i}{x_j} \, dx.
\end{aligned}
\end{multline*}
Hence there exists a constant $C$, depending only on $n$ and $\Omega$, such that
\begin{multline*}
\left|\int_{\partial \Omega} \left(|Du_{p, \epsilon}|^2 + \epsilon^2\right)^{\frac{p}{2} - 1} \left(|Du_{p, \epsilon}|^2 + \epsilon^2 - p |D_\nu u_{p, \epsilon}|^2\right) \, d\Ha^{n - 1}\right| \\
\le Cp\int_\Omega H_{p, \epsilon}(|Du_{p, \epsilon}|) \, dx.
\end{multline*}
Write
\[
f_{p, \epsilon} = D_\nu u_{p, \epsilon} \quad \text{and} \quad g_{p, \epsilon} = |D'u_{p, \epsilon}| = \sqrt{|Du_{p, \epsilon}|^2 - |f_{p, \epsilon}|^2}
\]
on $\partial \Omega$. Then we may write the above inequality in the form
\begin{multline*}
\left|\int_{\partial \Omega} \left(|Du_{p, \epsilon}|^2 + \epsilon^2\right)^{\frac{p}{2} - 1} \left(g_{p, \epsilon}^2 + \epsilon^2 - (p - 1) |f_{p, \epsilon}|^2\right) \, d\Ha^{n - 1}\right| \\
\le Cp\int_\Omega H_{p, \epsilon}(|Du_{p, \epsilon}|) \, dx.
\end{multline*}
Hence
\begin{multline} \label{eq:boundary1}
\int_{\partial \Omega} \left(|Du_{p, \epsilon}|^2 + \epsilon^2\right)^{\frac{p}{2} - 1} |f_{p, \epsilon}|^2 \, d\Ha^{n - 1} \\
\begin{aligned}
& \le \frac{1}{p - 1} \int_{\partial \Omega} \left(|Du_{p, \epsilon}|^2 + \epsilon^2\right)^{\frac{p}{2} - 1} (g_{p, \epsilon}^2 + \epsilon^2) \, d\Ha^{n - 1} \\
& \quad + \frac{Cp}{p - 1} \int_\Omega H_{p, \epsilon}(|Du_{p, \epsilon}|) \, dx,
\end{aligned}
\end{multline}
and
\[
\begin{split}
\int_{\partial \Omega} H_{p, \epsilon}(|Du_{p, \epsilon}|) \, d\Ha^{n - 1} & = \int_{\partial \Omega} \left(|Du_{p, \epsilon}|^2 + \epsilon^2\right)^{\frac{p}{2} - 1} (|f_{p, \epsilon}|^2 + g_{p, \epsilon}^2 + \epsilon^2) \, d\Ha^{n - 1} \\
& \le \frac{p}{p - 1} \int_{\partial \Omega} \left(|Du_{p, \epsilon}|^2 + \epsilon^2\right)^{\frac{p}{2} - 1} (g_{p, \epsilon}^2 + \epsilon^2) \, d\Ha^{n - 1} \\
& \quad + \frac{Cp}{p - 1} \int_\Omega H_{p, \epsilon}(|Du_{p, \epsilon}|) \, dx \\
& \le \frac{p - 2}{p - 1} \int_{\partial \Omega} H_{p, \epsilon}(|Du_{p, \epsilon}|) \, d\Ha^{n - 1} \\
& \quad + \frac{2}{p - 1} \int_{\partial \Omega} \left(g_{p, \epsilon}^2 + \epsilon^2\right)^{p/2} \, d\Ha^{n - 1} \\
& \quad + \frac{Cp}{p - 1} \int_\Omega H_{p, \epsilon}(|Du_{p, \epsilon}|) \, dx.
\end{split}
\]
Here we have used Young's inequality in the last step. It follows that
\begin{equation} \label{eq:boundary-finite-p}
\int_{\partial \Omega} H_{p, \epsilon}(|Du_{p, \epsilon}|) \, d\Ha^{n - 1} \le 2 \int_{\partial \Omega} \left(g_{p, \epsilon}^2 + \epsilon^2\right)^{p/2} \, d\Ha^{n - 1} + Cp \int_\Omega H_{p, \epsilon}(|Du_{p, \epsilon}|) \, dx.
\end{equation}
Feeding this back into \eqref{eq:boundary1}, we obtain
\begin{equation} \label{eq:boundary2}
\begin{split}
\int_{\partial \Omega} \left(|Du_{p, \epsilon}|^2 + \epsilon^2\right)^{\frac{p}{2} - 1} |f_{p, \epsilon}|^2 \, d\Ha^{n - 1} & \le \frac{2}{p - 1} \int_{\partial \Omega} \left(g_{p, \epsilon}^2 + \epsilon^2\right)^{p/2} \, d\Ha^{n - 1} \\
& \quad + \frac{2Cp}{p - 1} \int_\Omega H_{p, \epsilon}(|Du_{p, \epsilon}|) \, dx.
\end{split}
\end{equation}

Define
\[
\Gamma_{p, \epsilon}^1 = \set{x \in \partial \Omega}{|f_{p, \epsilon}(x)|^2 \le e_p^2 - c^2}
\]
and $\Gamma_{p, \epsilon}^2 = \partial \Omega \setminus \Gamma_{p, \epsilon}^1$.
By the choice of the boundary data, we have $g_{p, \epsilon}^2 + \epsilon^2 \le c^2$ when $\epsilon$ is small, and thus
$|Du_{p, \epsilon}|^2 + \epsilon^2 \le e_p^2$ on $\Gamma_{p, \epsilon}^1$. Hence
\[
\int_{\Gamma_{p, \epsilon}^1} H_{p, \epsilon}(|Du_{p, \epsilon}|) \, d\Ha^{n - 1} \le e_p^p \Ha^{n - 1}(\partial \Omega).
\]
On $\Gamma_{p, \epsilon}^2$, we note that
\[
|f_{p, \epsilon}|^2 \ge \frac{e_p^2 - c^2}{c^2} (g_{p, \epsilon}^2 + \epsilon^2)
\]
for sufficiently small values of $\epsilon$. Therefore,
\[
|Du_{p, \epsilon}|^2 + \epsilon^2 = |f_{p, \epsilon}|^2 + g_{p, \epsilon}^2 + \epsilon^2 \le \frac{e_p^2}{e_p^2 - c^2} |f_{p, \epsilon}|^2
\]
on $\Gamma_{p, \epsilon}^2$. Hence
\[
\int_{\Gamma_{p, \epsilon}^2} H_{p, \epsilon}(|Du_{p, \epsilon}|) \, d\Ha^{n - 1} \le \frac{e_p^2}{e_p^2 - c^2} \int_{\partial \Omega} \left(|Du_{p, \epsilon}|^2 + \epsilon^2\right)^{\frac{p}{2} - 1} |f_{p, \epsilon}|^2 \, d\Ha^{n - 1}.
\]
Using \eqref{eq:boundary2}, we see that
\begin{multline*}
\int_{\Gamma_{p, \epsilon}^2} H_{p, \epsilon}(|Du_{p, \epsilon}|) \, d\Ha^{n - 1} \\
\begin{aligned}
& \le \frac{2e_p^2}{(p - 1)(e_p^2 - c^2)} \left(\int_{\partial \Omega} \left(g_{p, \epsilon}^2 + \epsilon^2\right)^{p/2} \, d\Ha^{n - 1} + Cp \int_\Omega H_{p, \epsilon}(|Du_{p, \epsilon}|) \, dx\right) \\
& \le \frac{2e_p^2}{(p - 1)(e_p^2 - c^2)} \left(c^p \Ha^{n - 1}(\partial \Omega) + Cp \int_\Omega H_{p, \epsilon}(|Du_{p, \epsilon}|) \, dx\right).
\end{aligned}
\end{multline*}
Since we know that $e_p \to e_\infty$ as $p \to \infty$, it follows that there exists a constant $C'$,
depending on $n$, $\Omega$, $e_\infty$, and $c$, such that
\begin{equation} \label{eq:boundary3}
\int_{\partial \Omega} H_{p, \epsilon}(|Du_{p, \epsilon}|) \, d\Ha^{n - 1} \le C'\left(e_p^p + \int_\Omega H_{p, \epsilon}(|Du_{p, \epsilon}|) \, dx\right),
\end{equation}
provided that $p$ is sufficiently large.

We now consider the limit $\epsilon \searrow 0$ (still for $p < \infty$ fixed).
Since $E_{p, \epsilon}(u_{p, \epsilon}) \le E_{p, \epsilon}(u_{0, \epsilon})$,
it is clear that the family of functions $u_{p, \epsilon}$ is bounded in $W^{1, p}(\Omega; \R^N)$.
We may therefore choose a sequence $\epsilon_k \searrow 0$ such that $u_{p, \epsilon_k} \rightharpoonup \tilde{u}_p$
weakly in this space. It is obvious that $\tilde{u}_p \in u_0 + W_0^{1, p}(\Omega; \R^N)$.

Let $v \in u_0 + W_0^{1, p}(\Omega; \R^N)$. Define $v_\epsilon = v + u_{0, \epsilon} - u_0$,
in order to obtain a function in $u_{0, \epsilon} + W_0^{1, p}(\Omega; \R^N)$. Then
\[
E_p(\tilde{u}_p) \le \liminf_{k \to \infty} E_{p, \epsilon_k}(u_{p, \epsilon_k}) \le \liminf_{k \to \infty} E_{p, \epsilon_k}(v_{\epsilon_k}) = E_p(v).
\]
Hence $\tilde{u}_p$ minimises $E_p$ in $u_0 + W_0^{1, p}(\Omega; \R^N)$. The strict convexity of
the functional implies that $\tilde{u}_p = u_p$. Inserting $v = u_p$ and using similar estimates,
we also see that
\[
\limsup_{k \to \infty} E_{p, \epsilon_k}(u_{p, \epsilon_k}) \le E_p(u_p).
\]
Hence the convergence $u_{p, \epsilon_k} \to u_p$ is strong in $W^{1, p}(\Omega; \R^N)$.

Consider the measures
\[
m_{p, \epsilon} = \frac{\left(|Du_{p, \epsilon}|^2 + \epsilon^2\right)^{\frac{p}{2} - 1}}{e_p^{p - 2} \Le^n(\Omega)} \Ha^{n - 1} \restr \partial \Omega.
\]
Using inequality \eqref{eq:boundary-finite-p}, and recalling that $g_{p, \epsilon}^2 + \epsilon^2 \le c$ when
$\epsilon$ is small, we see that
\[
\limsup_{\epsilon \searrow 0} \int_{\partial \Omega} \left(1 + |f_{p, \epsilon}|^2\right) \, dm_{p, \epsilon} < \infty
\]
for any fixed $p < \infty$.
Hence, by Proposition \ref{prp:weak-compactness}, we may assume that
the measure-function pairs $(m_{p, \epsilon_k}, f_{p, \epsilon_k})$ over $\partial \Omega$ converge
weakly to a measure-function pair $(m_p, f_p)$. Using \eqref{eq:Euler-Lagrange-regularised} and passing
to the limit, we obtain identity \eqref{eq:boundary}. Inequality \eqref{eq:boundary-measures} follows from \eqref{eq:boundary3}
and Proposition \ref{prp:weak-compactness}.
\end{proof}

We also briefly consider the case $e_\infty' = e_\infty$.
We conclude this section by proving Proposition \ref{prp:boundary-point}, thus showing that in this situation,
the minimum value of $E_\infty$ is dictated locally near a single boundary point.

\begin{proof}[Proof of Proposition \ref{prp:boundary-point}]
Suppose that $e_\infty' = e_\infty$.
Fix $u \in u_0 + W_0^{1, \infty}(\Omega; \R^N)$. We extend $u$ outside of $\Omega$ such that
it is Lipschitz continuous globally in $\R^n$.
Let $L$ denote the Lipschitz constant of this extension.

Define
\[
\alpha(x) = \lim_{r \searrow 0} \esssup_{\partial \Omega \cap B_r(x)} |D'u_0|
\]
for $x \in \partial \Omega$. This gives rise to an upper semicontinuous function with
\[
e_\infty = \sup_{x \in \partial \Omega} \alpha(x).
\]
Hence there exists $x \in \partial \Omega$ such that $e_\infty = \alpha(x)$. We fix this point now.

Let $\epsilon \in (0, \frac{1}{2}]$. Then there exists $r > 0$ such that
\[
\esssup_{\partial \Omega \cap B_r(x)} |D' u_0| > e_\infty - \epsilon.
\]
Since the restriction of $u_0$ to $\partial \Omega$
is Lipschitz continuous, it is differentiable almost everywhere
with respect to $\Ha^{n - 1}$. Hence we can find a point
$y \in \partial \Omega \cap B_r(x)$ such that $D'u_0(y)$ exists with
\[
|D' u_0(y)| \ge e_\infty - \epsilon.
\]
We may assume without loss of generality that $y = 0$ and $u_0(0) = 0$, and that the tangent space
of $\partial \Omega$ at $y$ is $\R^{n - 1} \times \{0\}$.
We write $A = D' u_0(y) = (a_{kj})_{1 \le k \le N, 1 \le j \le n - 1}$.
Then $|A| \le \sqrt{N}L$.

If $s > 0$ is sufficiently small, then
\[
\partial \Omega \cap [-s, s]^n \subseteq [-s, s]^{n - 1} \times [-\epsilon s, \epsilon s] \quad \text{and} \quad [-s, s]^{n - 1} \times [\epsilon s, s] \subseteq \Omega,
\]
while at the same time,
\[
|u_0(z') - Az'| \le \epsilon s
\]
for all $z' \in [-s, s]^n \cap \partial \Omega$. Set $\Sigma_\epsilon = [-s, s]^{n - 1} \times [\epsilon s, 2\epsilon s]$.
For any $z \in \Sigma_\epsilon$,
there exists $z' \in [-s, s]^n \cap \partial \Omega$ such that $|z - z'| \le 3\epsilon s$. Hence
\[
|u(z) - Az| \le |u(z) - u(z')| + |u_0(z') - Az'| + |A(z' - z)| \le \bigl(3(\sqrt{N} + 1)L + 1\bigr) \epsilon s
\]
for all $z \in \Sigma_\epsilon$.

Integrating $\dd{u_k}{x_j}$ along lines parallel to the $x_j$-axis, we see that there exists a constant
$C$, depending only on $n$, $N$, and $L$, such that for $1 \le k \le N$ and for $1 \le j \le n - 1$,
the inequality
\[
\left|\fint_{\Sigma_\epsilon} \dd{u_k}{x_j} \, dx - a_{kj}\right| \le C\epsilon
\]
holds true.
Thus there exists a subset of $\Sigma_\epsilon$ of positive measure where
$|Du| \ge |A| - C'\epsilon$ for another constant $C'$ depending only on $n$, $N$, and $L$.
(Otherwise we would conclude that
\[
\left|\fint_{\Sigma_\epsilon} Du \, dx - A\right| \ge |A| - \fint_{\Sigma_\epsilon} |Du| \, dx \ge C'\epsilon,
\]
using the reverse triangle inequality.) It follows that
\[
\esssup_{B_r(x)} |Du| \ge |A| - C'\epsilon \ge e_\infty - (C' + 1) \epsilon.
\]
As $\epsilon$ was chosen arbitrarily, this implies the claim.
\end{proof}

\section{Key estimates} \label{sct:key}

Throughout this section and the rest of the paper, the function $u_\infty$ and the
measure-function pair $M_\infty = (\mu_\infty, F_\infty)$ are as constructed in Section \ref{sct:measure-function}
and are fixed. Recall that $u_\infty$ is a minimiser of $E_\infty$.
If $e_\infty' < e_\infty$, then we consider the measure-function pairs
$(m_p, f_p)$ from Proposition \ref{prp:boundary-measures} as well. We may assume that
$(m_{p_k}, f_{p_k})$ converges weakly to a limiting measure-function pair $(m_\infty, f_\infty)$
over $\partial \Omega$, which will then satisfy
\begin{equation} \label{eq:boundary-measure}
\int_{\overline{\Omega}} F_\infty : D\phi \, d\mu_\infty = \int_{\partial \Omega} \phi \cdot f_\infty \, dm_\infty
\end{equation}
for all $\phi \in C_0^\infty(\R^n; \R^N)$. If $e_\infty' = e_\infty$, then we choose an arbitrary
measure-function pair $(m_\infty, f_\infty)$. Even then, we can still use equation \eqref{eq:no-boundary}
for $\phi \in C_0^\infty(\Omega; \R^N)$, and we conclude that \eqref{eq:boundary-measure} still holds true
for these test functions.

We now analyse $M_\infty$ in more detail. This will also reveal some information
about other possible minimisers of $E_\infty$. Some of the key arguments in this section are similar
to estimates due to Evans and Yu \cite{Evans-Yu:05}.

We prove the following statements.

\begin{theorem} \label{thm:limiting-measure-function-pair}
Suppose that $v \in u_0 + W_0^{1, \infty}(\Omega; \R^N)$ is a minimiser of
$E_\infty$.
\begin{enumerate}
\item \label{itm:minimisers-local}
Then $\locrep{Dv}{\mu_\infty} = F_\infty$,
and $|Dv|^\star(x) = |F_\infty(x)| = e_\infty$ for $\mu_\infty$-almost every $x \in \Omega$.
\item \label{itm:minimisers-global}
If $e_\infty' < e_\infty$, then $\rep{Dv}{\mu_\infty} = F_\infty$, and $|Dv|^\star(x) = |F_\infty(x)| = e_\infty$ for $\mu_\infty$-almost every $x \in \overline{\Omega}$.
\item \label{itm:strong-convergence-local}
The convergence $M_{p_\ell} \to M_\infty$ is strong in every compact subset of $\Omega$.
\item \label{itm:strong-convergence-global}
If $e_\infty' < e_\infty$, then the convergence $M_{p_\ell} \to M_\infty$ is strong in $\overline{\Omega}$.
\end{enumerate}
\end{theorem}

Before we prove these results, we note that statement \ref{itm:strong-convergence-local} has the following consequence.
If we write $F_\infty = (F_{1 \infty}, \dotsc, F_{N \infty})$, then \eqref{eq:Pohozaev} gives rise to
the identity
\begin{equation} \label{eq:stationary}
\sum_{k = 1}^N \int_\Omega F_{k \infty} \cdot D_{F_{k \infty}}\psi \, d\mu_\infty = 0
\end{equation}
for all $\psi \in C_0^\infty(\Omega; \R^n)$, owing to Proposition \ref{prp:weak-compactness}.
This equation complements \eqref{eq:no-boundary}, and we will use
it in Section \ref{sct:structure} to say something about the structure of $(M_\infty, F_\infty)$, although we
will formulate this in terms of a corresponding $N$-tuple of $1$-currents.

In the framework of currents, equation \eqref{eq:stationary} corresponds to \eqref{eq:geodesic}.
It is a weak formulation of the equation for geodesics and is one of key properties of the
measure-function pair $M_\infty$.

For the proof of the theorem, we require the following lemma.

\begin{lemma} \label{lem:mass-estimate}
For any $\xi \in C^0(\overline{\Omega})$ with $\xi \ge 0$, and for any $\alpha \in (0, 1)$,
\[
\alpha^2 e_p^2 \int_\Omega \xi \, d\mu_p \le \int_\Omega \xi |D u_p|^2 \, d \mu_p + \alpha^p e_p^2 \|\xi\|_{C^0(\Omega)}.
\] 
\end{lemma}

\begin{proof}
Given $\alpha \in (0, 1)$, define $A_p = \set{x \in \Omega}{|D u_p| \le \alpha e_p}$.
We first note that
\[
\mu_p(A_p) = \frac{e_p^{2 - p}}{\mathcal{L}^n(\Omega)} \int_{A_p} |D u_p|^{p - 2} \, dx \le \alpha^{p - 2} \frac{\mathcal{L}^n(A_p)}{\mathcal{L}^n(\Omega)} \le \alpha^{p - 2}.
\]
Hence
\[
\begin{split}
\int_\Omega \xi |D u_p|^2 \, d\mu_p & \ge \int_{\Omega \setminus A_p} \xi |D u_p|^2 \, d\mu_p \\
& \ge \alpha^2 e_p^2 \int_{\Omega \setminus A_p} \xi \, d\mu_p \\
& = \alpha^2 e_p^2 \left(\int_\Omega \xi \, d\mu_p - \int_{A_p} \xi \, d\mu_p\right) \\
& \ge \alpha^2 e_p^2 \left(\int_\Omega \xi \, d\mu_p - \|\xi\|_{C^0(\Omega)} \mu_p(A_p)\right) \\
& \ge \alpha^2 e_p^2 \int_\Omega \xi \, d\mu_p - \alpha^p e_p^2 \|\xi\|_{C^0(\Omega)}.
\end{split}
\]
This implies the desired inequality.
\end{proof}

\begin{proof}[Proof of Theorem \ref{thm:limiting-measure-function-pair}.]
We first prove the local statements \ref{itm:minimisers-local} and \ref{itm:strong-convergence-local},
allowing for the possibility that $e_\infty' = e_\infty$.
Let $\eta \in C_0^\infty(B_1(0))$ with $\int_{\R^n} \eta \, dx = 1$. Set $\eta_\epsilon(x) = \epsilon^{-n} \eta(x/\epsilon)$ and define $v_\epsilon = v * \eta_\epsilon$ (where for convenience $v$ is extended
arbitrarily outside of $\Omega$, so that $v_\epsilon$ is
well-defined in $\Omega$). Choose $\xi \in C_0^\infty(\Omega)$ with $0 \le \xi \le 1$. Also choose
$\alpha \in (0, 1)$. Then
\begin{multline*}
\int_\Omega \xi |D v_\epsilon - D u_p|^2 \, d\mu_p \\
\begin{aligned}
& = \int_\Omega \xi \bigl(|D v_\epsilon|^2 - |D u_p|^2\bigr) \, d\mu_p + 2 \int_\Omega \xi (D u_p - D v_\epsilon) : D u_p \, d\mu_p \\
& = \int_\Omega \xi \bigl(|D v_\epsilon|^2 - |D u_p|^2\bigr) \, d\mu_p - 2 \int_\Omega (u_p - v_\epsilon) \otimes D \xi : D u_p \, d\mu_p \\
& \le \int_\Omega \xi |D v_\epsilon|^2 \, d\mu_p - \alpha^2 e_p^2 \int_\Omega \xi \, d\mu_p + \alpha^p e_p^2 - 2 \int_\Omega (u_p - v_\epsilon) \otimes D \xi : D u_p \, d\mu_p.
\end{aligned}
\end{multline*}
Here we have used Lemma \ref{lem:mass-estimate} in the last step. Hence
\begin{multline*}
\limsup_{k \to \infty} \int_\Omega \xi |D v_\epsilon - D u_{p_k}|^2 \, d\mu_{p_k} \le \int_\Omega \xi |D v_\epsilon|^2 \, d\mu_\infty - \alpha^2 e_\infty^2 \int_\Omega \xi \, d\mu_\infty \\
- 2 \int_\Omega (u_\infty - v_\epsilon) \otimes D \xi : F_\infty \, d\mu_\infty.
\end{multline*}

Let $r > 0$. Choose $\xi$ such that $\xi \equiv 1$ in $\Omega \setminus \Omega_r$ and $\xi \equiv 0$ in $\Omega_{r/2}$,
and such that $|D \xi| \le 4/r$ in all of $\Omega$. Note that
\[
\|u_\infty - v\|_{C^0(\overline{\Omega_r})} \le 2\sqrt{N} r e_\infty,
\]
because $v$ is a minimiser of $E_\infty$ and thus $\|D v\|_{L^\infty(\Omega)} = \|D u_\infty\|_{L^\infty(\Omega)} = e_\infty$.
If $\epsilon < r$, then there exists a constant $C$, depending only on $N$, $n$, and $\eta$, such that
\[
\|u_\infty - v_\epsilon\|_{C^0(\overline{\Omega_r})} \le Cr e_\infty.
\]
It follows that
\[
-2\int_\Omega (u_\infty - v_\epsilon) \otimes D \xi : F_\infty \, d\mu_\infty \le 8C e_\infty \left(\mu_\infty(\Omega_r) \int_\Omega |F_\infty|^2 \, d\mu_\infty\right)^{1/2}.
\]

Given a compact set $K \subset \Omega$ and given $\gamma > 0$, we may choose $r > 0$ such that
$K \cap \Omega_r = \emptyset$ and 
\[
8C e_\infty \left(\mu_\infty(\Omega_r) \int_\Omega |F_\infty|^2 \, d\mu_\infty\right)^{1/2} \le \gamma.
\]
For the above choice of $\xi$, we then conclude that
\begin{equation} \label{eq:v_epsilon}
\limsup_{k \to \infty} \int_\Omega \xi |D v_\epsilon - D u_{p_k}|^2 \, d\mu_{p_k} \le \int_\Omega \xi |D v_\epsilon|^2 \, d\mu_\infty - \alpha^2 e_\infty^2 \int_\Omega \xi \, d\mu_\infty + \gamma
\end{equation}
for all $\epsilon \le r$.
Clearly $|D v_\epsilon| \le \|D v\|_{L^\infty(\Omega)} \le e_\infty$ everywhere in $\Omega$. Therefore,
\begin{equation} \label{eq:v_epsilon1}
\limsup_{k \to \infty} \int_K |D v_\epsilon - D u_{p_k}|^2 \, d\mu_{p_k} \le (1 - \alpha^2) e_\infty^2 \int_\Omega \xi \, d\mu_\infty + \gamma \le (1 - \alpha^2) e_\infty^2 + \gamma.
\end{equation}
By Proposition \ref{prp:weak-compactness},
\[
\int_K |D v_\epsilon - F_\infty|^2 \, d\mu_\infty \le (1 - \alpha^2) e_\infty^2 + \gamma
\]
as well. Since $\alpha \in (0, 1)$ and $\gamma > 0$ may be chosen arbitrarily here (and then the inequality holds for
$\epsilon$ small enough, depending on $\gamma$), it follows that
\[
\lim_{\epsilon \searrow 0} \int_K |D v_\epsilon - F_\infty|^2 \, d\mu_\infty = 0.
\]
According to Definition \ref{def:restriction-local}, this means that $F_\infty = \locrep{Dv}{\mu_\infty}$.

Because $|D v_\epsilon| \le e_\infty$ everywhere, this locally strong $L^2$-convergence with respect
to $\mu_\infty$ also implies that $|F_\infty| \le e_\infty$ at $\mu_\infty$-almost every point.
Now we note that \eqref{eq:v_epsilon} further gives rise to the inequality
\begin{equation} \label{eq:v_epsilon2}
\int_\Omega \xi |F_\infty|^2 \, d\mu_\infty = \lim_{\epsilon \searrow 0} \int_\Omega \xi |D v_\epsilon|^2 \, d\mu_\infty \ge e_\infty^2 \int_\Omega \xi \, d\mu_\infty.
\end{equation}
Hence $|F_\infty| = e_\infty$ at $\mu_\infty$-almost every point.

Inequality \eqref{eq:v_epsilon2}, together with the dominated convergence theorem, also has the
consequence that
\begin{equation} \label{eq:|Dv|*}
\begin{split}
\int_\Omega \xi (|D v|^\star)^2 \, d\mu_\infty & = \int_\Omega \xi \lim_{\epsilon \searrow 0} \esssup_{B_\epsilon(x)} |Dv|^2 \, d\mu_\infty(x) \\
& = \lim_{\epsilon \searrow 0} \int_\Omega \xi \esssup_{B_\epsilon(x)} |Dv|^2 \, d\mu_\infty(x) \\
& \ge \lim_{\epsilon \searrow 0} \int_\Omega \xi |Dv_\epsilon|^2 \, d\mu_\infty \\
& \ge e_\infty^2 \int_\Omega \xi \, d\mu_\infty.
\end{split}
\end{equation}
Since we clearly have the pointwise bound $|Dv|^\star \le e_\infty$, we conclude
that $|D v|^\star = e_\infty$ at $\mu_\infty$-almost every point.
All the claims of statement \ref{itm:minimisers-local} are now proved.

The claim of statement \ref{itm:strong-convergence-local} relies on the same inequalities.
We have shown, as a consequence of \eqref{eq:v_epsilon1},
that for any compact set $K \subset \Omega$, if $\delta > 0$ is given, then
there exists $\epsilon_0 > 0$ such that for any $\epsilon \in (0, \epsilon_0]$,
\begin{equation} \label{eq:key-estimate}
\limsup_{k \to \infty} \int_K |D v_\epsilon - D u_{p_k}|^2 \, d\mu_{p_k} \le \delta.
\end{equation}
But clearly, for a fixed $\epsilon$,
\begin{multline*}
\limsup_{k \to \infty} \int_K |D v_\epsilon - D u_{p_k}|^2 \, d\mu_{p_k} \\
\begin{aligned}
& = \limsup_{k \to \infty} \int_K \bigl(|D v_\epsilon|^2 - 2D v_\epsilon : D u_{p_k} + |D u_{p_k}|^2\bigr) \, d\mu_{p_k} \\
& = \int_K \bigl(|D v_\epsilon|^2 - 2D v_\epsilon : F_\infty\bigr) \, d\mu_\infty + \limsup_{k \to \infty} \int_K |D u_{p_k}|^2 \, d\mu_{p_k} \\
& = \int_K |D v_\epsilon - F_\infty|^2 \, d\mu_\infty - \int_K |F_\infty|^2 \, d\mu_\infty + \limsup_{k \to \infty} \int_K |D u_{p_k}|^2 \, d\mu_{p_k}.
\end{aligned}
\end{multline*}
Therefore, the above inequality \eqref{eq:key-estimate} implies that
\begin{equation} \label{eq:strong-convergence}
\limsup_{k \to \infty} \int_K |D u_{p_k}|^2 \, d\mu_{p_k} \le \int_K |F_\infty|^2 \, d\mu_\infty.
\end{equation}
By Proposition \ref{prp:weak-compactness}, this means that
$M_{p_k} \to M_\infty$ strongly in $K$.

Finally, we assume that $e_\infty' < e_\infty$, and we prove the global statements
\ref{itm:minimisers-global} and \ref{itm:strong-convergence-global} with global variants of the above
arguments. To this end, consider a
regular mollifier $\mathcal{M}_\epsilon$. We now define $v_\epsilon = \mathcal{M}_\epsilon v$.
We choose $\alpha \in (0, 1)$ again. Just as before, we compute
\begin{multline*}
\int_\Omega |D v_\epsilon - D u_p|^2 \, d\mu_p \\
\begin{aligned}
& = \int_\Omega \bigl(|D v_\epsilon|^2 - |D u_p|^2\bigr) \, d\mu_p + 2\int_{\partial \Omega} (u_p - v_\epsilon) \cdot f_p \, dm_p \\
& \le \int_\Omega |D v_\epsilon|^2 \, d\mu_p - \alpha^2 e_p^2 \mu_p(\Omega) + \alpha^p e_p^2 + 2\int_{\partial \Omega} (u_p - v_\epsilon) \cdot f_p \, dm_p.
\end{aligned}
\end{multline*}
Hence
\[
\begin{split}
\limsup_{k \to \infty} \int_\Omega |D v_\epsilon - D u_{p_k}|^2 \, d\mu_{p_k} & \le \int_{\overline{\Omega}} |D v_\epsilon|^2 \, d\mu_\infty - \alpha^2 e_\infty^2 \mu_\infty(\overline{\Omega}) \\
& \quad + 2 \int_{\partial \Omega} (u_\infty - v_\epsilon) \cdot f_\infty \, dm_\infty.
\end{split}
\]
Fix $\gamma > 0$. We know that $v_\epsilon \to v$ uniformly in $\overline{\Omega}$, and $v = u_\infty$ on $\partial \Omega$.
Hence
\[
2 \int_{\partial \Omega} (u_\infty - v_\epsilon) \cdot f_\infty \, dm_\infty \le \gamma
\]
whenever $\epsilon$ is sufficiently small. Therefore,
\begin{equation} \label{eq:v_epsilon-second}
\limsup_{k \to \infty} \int_\Omega |D v_\epsilon - D u_{p_k}|^2 \, d\mu_{p_k} \le \int_{\overline{\Omega}} |D v_\epsilon|^2 \, d\mu_\infty - \alpha^2 e_\infty^2 \mu_\infty(\overline{\Omega}) + \gamma.
\end{equation}

By the properties of regular mollifiers, we know that
\begin{equation} \label{eq:C^0-bound-v_epsilon}
\limsup_{\epsilon \searrow 0} \|Dv_\epsilon\|_{C^0(\overline{\Omega})} \le \|D v\|_{L^\infty(\Omega)} \le e_\infty.
\end{equation}
Hence
\[
\limsup_{k \to \infty} \int_\Omega |D v_\epsilon - D u_{p_k}|^2 \, d\mu_{p_k} \le (1 - \alpha^2) e_\infty^2 \mu_\infty(\overline{\Omega}) + 2\gamma
\]
when $\epsilon$ is small enough. Proposition \ref{prp:weak-compactness} then implies that
\[
\lim_{\epsilon \searrow 0} \int_{\overline{\Omega}} |D v_\epsilon - F_\infty|^2 \, d\mu_\infty = 0.
\]
It follows that $F_\infty = \rep{Dv}{\mu_\infty}$.

Using \eqref{eq:C^0-bound-v_epsilon} again, we conclude that the inequality
$|F_\infty| \le e_\infty$ holds $\mu_\infty$-almost everywhere in $\overline{\Omega}$.
By \eqref{eq:v_epsilon-second},
\[
\int_{\overline{\Omega}} |F_\infty|^2 \, d\mu_\infty = \lim_{\epsilon \searrow 0} \int_{\overline{\Omega}} |D v_\epsilon|^2 \, d\mu_\infty \ge e_\infty^2 \mu_\infty(\overline{\Omega}).
\]
Hence $|F_\infty| = e_\infty$ at $\mu_\infty$-almost every point of $\overline{\Omega}$.

As in \eqref{eq:|Dv|*}, we see that
\[
\begin{split}
\int_{\overline{\Omega}} (|D v|^\star)^2 \, d\mu_\infty & = \lim_{\epsilon \searrow 0} \int_{\overline{\Omega}} \esssup_{B_\epsilon(x) \cap \Omega} |Dv|^2 \, d\mu_\infty(x) \\
& \ge\lim_{\epsilon \searrow 0} \int_{\overline{\Omega}} |Dv_\epsilon|^2 \, d\mu_\infty \\
& = \lim_{\epsilon \searrow 0} \int_{\overline{\Omega}} |F_\infty|^2 \, d\mu_\infty \\
& = e_\infty^2 \mu_\infty(\overline{\Omega})
\end{split}
\]
by the properties of $\mathcal{M}_\epsilon$.
Hence $|D v|^\star = e_\infty$ at $\mu_\infty$-almost every point of $\overline{\Omega}$.
This completes the proof of \ref{itm:minimisers-global}.

For the proof of \ref{itm:strong-convergence-global}, we observe that
the derivation of \eqref{eq:strong-convergence} now works for $K = \overline{\Omega}$ as well.
Thus it suffices to use Proposition \ref{prp:weak-compactness} again.
\end{proof}

\section{Combining the tools} \label{sct:proof}

In this section, we prove Theorem \ref{thm:main}. If $e_\infty = 0$, then we can
choose any unit vector $\vec{T}_0 \in \R^{N \times n}$ and set $T = [\Le^n \restr \overline{\Omega}, \vec{T}_0]$.
The statements of the theorem will then be satisfied trivially. We therefore assume that
$e_\infty > 0$ henceforth.

The current $T$ in the theorem
is then just another representation of the measure-function pair $M_\infty = (\mu_\infty, F_\infty)$ constructed above.
Namely, we set $T = [\mu_\infty, F_\infty]$. Since we know that $|F_\infty(x)| = e_\infty$ for $\mu_\infty$-almost
every $x \in \Omega$ by Theorem \ref{thm:limiting-measure-function-pair}, this means that
$\|T\| \restr \Omega = e_\infty \mu_\infty \restr \Omega$ and $\vec{T} = e_\infty^{-1} F_\infty$ in $\Omega$.
If $e_\infty' < e_\infty$, then $\|T\| = e_\infty \mu_\infty$ and $\vec{T} = e_\infty^{-1} F_\infty$ on $\overline{\Omega}$
by the same theorem.

Theorem \ref{thm:main} now follows from the
results in the previous sections. We will give the details below.
First, however, we formulate and prove a local version of the mass minimising property in statement
\ref{itm:non-trivial}.\ref{itm:geodesic} in Theorem \ref{thm:main}. This result also holds true if $e_\infty' = e_\infty$.

\begin{theorem} \label{thm:local-geodesic}
Suppose that $S$ is an $N$-tuple of $1$-currents with locally finite mass and with
$\supp \partial S \cap \Omega = \emptyset$. If $K \subset \Omega$ is a compact set such that $S(\omega) = T(\omega)$ for all $\omega \in (\D^1(\Omega))^N$ with
$\omega = 0$ in $K$, then $\|T\|(K) \le \|S\|(K)$.
\end{theorem}

\begin{proof}
We first note that $\locrep{Du_\infty}{\|T\|} = e_\infty \vec{T}$ by Theorem \ref{thm:limiting-measure-function-pair}.
According to Corollary \ref{cor:optimal-current-local}, this implies that
\[
-T(u_\infty d\chi) = e_\infty \int_\Omega \chi \, d\|T\|
\]
for any $\chi \in C_0^\infty(\Omega)$ with $\chi \ge 0$. Hence the claim follows from Corollary \ref{cor:minimising-current-local}.
\end{proof}

For the proof of Theorem \ref{thm:main}, we also require some additional tools, including the following estimate.

\begin{lemma} \label{lem:decay-near-boundary}
Suppose that $e_\infty' < e_\infty$. Then there exist $R > 0$ and $\beta \in (0, 1)$ such that $\|T\|(\Omega_{r/2} \cup \partial \Omega) \le \beta \|T\|(\Omega_r \cup \partial \Omega)$ for all $r \in (0, R]$.
\end{lemma}

\begin{proof}
Choose $R > 0$ such that there exists a smooth nearest point projection
$\varpi \colon \Omega_{2R} \to \partial \Omega$.
We may replace $u_0$ (while still using the same notation) by a function such that
$u_0(x) = u_0(\varpi(x))$ for $x \in \Omega_R$. Then
\[
c \coloneqq \esssup_{\Omega_R} |D u_0| < e_\infty,
\]
provided that $R$ is sufficiently small.

Now fix $r \in (0, R]$. Choose $\xi \in C_0^\infty(\R^n)$ with $0 \le \xi \le 1$ and such that
$\xi \equiv 1$ in $\Omega_{r/2}$ and $\xi \equiv 0$ in a neighbourhood of $\Omega \setminus \Omega_r$, and such that $|D\xi| \le 4/r$. Then
\begin{multline*}
\int_\Omega \xi |D u_p|^p \, dx \\
\begin{aligned}
& = \int_\Omega \xi |D u_p|^{p - 2} D u_p : D u_0 \, dx + \int_\Omega \xi |D u_p|^{p - 2} D u_p : (D u_p - D u_0) \, dx \\
& = \int_\Omega \xi |D u_p|^{p - 2} D u_p : D u_0 \, dx - \int_\Omega |D u_p|^{p - 2} (u_p - u_0) \otimes D \xi : D u_p \, dx.
\end{aligned}
\end{multline*}
Moreover, given $\alpha \in (0, 1)$, we can estimate
\[
\int_\Omega \xi |D u_p|^{p - 2} D u_p : D u_0 \, dx \le \frac{p - 1}{p} \alpha^{\frac{p}{p - 1}} \int_\Omega \xi |D u_p|^p \, dx + \frac{1}{p \alpha^p} \int_\Omega \xi |D u_0|^p \, dx
\]
by Young's inequality. Hence
\begin{multline*}
\left(1 - \frac{p - 1}{p} \alpha^{\frac{p}{p - 1}}\right) \int_\Omega \xi |D u_p|^p \, dx \\
\begin{aligned}
& \le \frac{1}{p \alpha^p} \int_\Omega \xi |D u_0|^p \, dx - \int_\Omega |D u_p|^{p - 2} (u_p - u_0) \otimes D \xi : D u_p \, dx \\
& \le \frac{c^p}{p \alpha^p} \int_\Omega \xi \, dx - \int_\Omega |D u_p|^{p - 2} (u_p - u_0) \otimes D \xi : D u_p \, dx.
\end{aligned}
\end{multline*}
In terms of the measures $\mu_p$, this means that
\begin{multline*}
\left(1 - \frac{p - 1}{p} \alpha^{\frac{p}{p - 1}}\right) \int_{\Omega} \xi |Du_p|^2 \, d\mu_p \\
\le \frac{c^p e_p^{2 - p}}{p \alpha^p} \fint_\Omega \xi \, dx - \int_\Omega (u_p - u_0) \otimes D \xi : D u_p \, d\mu_p.
\end{multline*}
Choose $\alpha > \frac{c}{e_\infty}$.
Restricting to $p_k$ and letting $k \to \infty$, we then find that
\[
(1 - \alpha) \int_{\overline{\Omega}} \xi |F_\infty|^2\, d\mu_\infty \le \int_\Omega (u_\infty - u_0) \otimes D \xi : F_\infty \, d\mu_\infty
\]
by Proposition \ref{prp:weak-compactness}.
Since $u_\infty - u_0 \in W_0^{1, \infty}(\Omega; \R^N)$,
there exists a constant $C_1$ such that
$|u_\infty - u_0| \le C_1 r$ in $\Omega_r$. Hence
\[
\int_{\Omega_{r/2} \cup \partial \Omega} |F_\infty|^2\, d\mu_\infty \le C_2 e_\infty \int_{\Omega_r \setminus \Omega_{r/2}} |F_\infty| \, d\mu_\infty,
\]
where $C_2 = 4C_1/((1 - \alpha) e_\infty)$. By Theorem \ref{thm:limiting-measure-function-pair} and by the
definition of $T$, this means that
\[
\|T\|(\Omega_{r/2} \cup \partial \Omega) \le C_2 \|T\|(\Omega_r \setminus \Omega_{r/2}).
\]
Adding $C_2 \|T\|(\Omega_{r/2} \cup \partial \Omega)$ on both sides of the
inequality, we conclude that
\[
\|T\|(\Omega_{r/2} \cup \partial \Omega) \le \frac{C_2}{C_2 + 1} \|T\|(\Omega_r \cup \partial \Omega).
\]
Thus we have the desired inequality for $\beta = C_2/(C_2 + 1)$.
\end{proof}

We now have everything in place for the proof of our main theorem.

\begin{proof}[Proof of Theorem \ref{thm:main}]
Recall that we use the assumption $e_\infty > 0$, as the theorem is trivial otherwise.
We have a measure-function pair $M_\infty = (\mu_\infty, F_\infty)$, constructed in Section \ref{sct:measure-function}, that satisfies
statements \ref{itm:minimisers-local} and \ref{itm:strong-convergence-local} of Theorems \ref{thm:limiting-measure-function-pair},
and \ref{itm:minimisers-global} and \ref{itm:strong-convergence-global} as well if $e_\infty' < e_\infty$.
As described at the beginning of this section, we set $T = [\mu_\infty, F_\infty]$.
Then Theorem \ref{thm:limiting-measure-function-pair} implies that $\|T\| \restr \Omega = e_\infty \mu_\infty \restr \Omega$ and
$\vec{T} = e_\infty^{-1} F_\infty$ in $\Omega$. If $e_\infty' < e_\infty$, then
$\|T\| = e_\infty \mu_\infty$ and $\vec{T} = e_\infty^{-1} F_\infty$ on $\overline{\Omega}$.
We conclude that $T \neq 0$ in the second case.

We now prove the individual statements of Theorem \ref{thm:main}.

\begin{subproof}{\ref{itm:boundary}}
It is clear that $\supp T \subseteq \overline{\Omega}$. For any $\sigma \in (\D^0(\R^n)^N$ with
compact support in $\Omega$, we note that
\[
T(d\sigma) = \int_\Omega F_\infty : D\sigma \, d\mu_\infty = 0
\]
by \eqref{eq:no-boundary}.
Hence $\supp \partial T \subseteq \partial \Omega$.
\end{subproof}

\begin{subproof}{\ref{itm:minimiser}}
Both statements follow immediately from Theorem \ref{thm:limiting-measure-function-pair}.
\end{subproof}

\begin{subproof}{\ref{itm:unique}}
We note that identity \eqref{eq:stationary} for the measure-function pair $M_\infty$
corresponds to \eqref{eq:geodesic} for $T$. Therefore,
if we have two minimisers $u, v \in u_0 + W_0^{1, \infty}(\Omega; \R^N)$ of $E_\infty$, then we
may apply Proposition \ref{prp:pointwise} to the function $u - v$. Since we have already proved
that $\locrep{Du}{\|T\|} = e_\infty \vec{T} = \locrep{Dv}{\|T\|}$, it follows that
$u = v$ on $\supp T$.
\end{subproof}

\begin{subproof}{\ref{itm:non-trivial}}
Here we assume that $e_\infty' < e_\infty$.
Then statement \ref{itm:normal} is a consequence of Proposition \ref{prp:boundary-measures}, as it follows that
\[
\partial T(\sigma) = \lim_{\ell \to \infty} \int_\Omega Du_{p_\ell} : D\sigma \, d\mu_{p_\ell} = \int_{\partial \Omega} f_{p_\ell} \cdot \sigma \, dm_{p_\ell} = \int_{\partial \Omega} f_\infty \cdot \sigma \, dm_\infty
\]
for all $\sigma \in (\D^0(\R^n)^N$. Here $(m_\infty, f_\infty)$ is the measure-function pair over $\partial \Omega$
found at the beginning of Section \ref{sct:key}.

Statement \ref{itm:trivial-on-boundary} is an obvious consequence of Lemma \ref{lem:decay-near-boundary}
and the fact that $T \neq 0$.

For the proof of statement \ref{itm:optimal}, we invoke Theorem \ref{thm:limiting-measure-function-pair}.
We conclude that $\rep{Du}{\|T\|} = e_\infty \vec{T}$.
Hence Proposition \ref{prp:optimal-current-global} implies that
$\partial T(u_0) = e_\infty \M(T)$.

Finally, we prove \ref{itm:geodesic}.
Recall that we have already proved the identity $\partial T(u_0) = e_\infty \M(T)$ when $e_\infty' < e_\infty$. Thus
Corollary \ref{cor:minimal-current-global} implies the desired statement.
\end{subproof}
\end{proof}

\section{The structure of $T$} \label{sct:structure}

In this final section we give some more results about the structure of the
$N$-tuple of $1$-currents $T$ constructed above. These are based on the
condition \eqref{eq:geodesic} and are closely related to standard
results on varifolds. Most of the results in the literature, however, do not
apply to $T$, as it is somewhat unusual: it should be thought of as a $1$-dimensional object,
because it acts on $1$-forms, but need not actually be $1$-dimensional in any other sense.
In contrast, most results on currents or varifolds in the literature assume rectifiability
in the appropriate dimension.

We use some more concepts from geometric measure theory in the following, including
the notion of a countably rectifiable measure. A definition can be found, e.g., in a book
by Mattila \cite[Definition 16.6]{Mattila:95}.

\begin{theorem} \label{thm:structure}
The measure $\|T\| \restr \Omega$ is absolutely continuous with respect to the one-dimensional Hausdorff measure $\mathcal{H}^1$.
Moreover, for any $\|T\|$-measurable set $A \subseteq \Omega$
such that $\mathcal{H}^1(A) < \infty$, the restriction $\|T\| \restr A$ is a countably $1$-rectifiable
measure. At $\|T\|$-almost every point $x \in A$, the approximate tangent space of $\|T\| \restr A$
contains $\{\vec{T}_1(x), \dotsc, \vec{T}_N(x)\}$.
\end{theorem}

This means that the dimension of $T$, say in the sense of Hausdorff dimension
for the measure $\|T\|$ as defined by Mattila, Mor\'an, and Rey \cite{Mattila-Moran-Rey:00}, is at least $1$
(unless $\|T\|(\Omega) = 0$). It can, however,
be higher. (An example is discussed in the introduction.) If we restrict $\|T\|$ to a
$1$-dimensional set $A$, then we have the structure typically assumed in geometric measure theory.
We then also find that the vectors $\vec{T}_1(x), \dotsc, \vec{T}_N(x)$ are tangential.
Since we have a $1$-dimensional approximate tangent space here, this means that they are
parallel to one another (or vanish) at $\|T\|$-almost every point $x \in A$.

For the proof, we use a monotonicity identity that is a standard tool in the theory of varifolds
(see \cite[\S 17]{Simon:83}). It is difficult, however, to find a formulation in the literature
for anything more general than a rectifiable varifold, even though the standard arguments apply
more generally. We therefore provide a proof for the convenience of the reader. 
As mentioned earlier, varifolds and currents are closely related, and we can formulate the
identity in terms of $T$.

\begin{lemma} \label{lem:monotonicity}
For any $x_0 \in \Omega$ and for $0 < s < r \le \dist(x_0, \partial \Omega)$,
the identity
\[
\frac{\|T\|(B_r(x_0))}{r} - \frac{\|T\|(B_s(x_0))}{s} = \int_{B_r(x_0) \setminus B_s(x_0)} \frac{1 - \bigl|\frac{x - x_0}{|x - x_0|} \cdot \vec{T}\bigr|^2}{|x - x_0|} \, d\|T\|
\]
holds true.
\end{lemma}

\begin{proof}
We may assume without loss of generality that $x_0 = 0$.

Choose a non-increasing function $\xi \in C^\infty(\R)$ with $\xi \equiv 1$ in $(-\infty, 0]$ and $\xi \equiv 0$ in
$[1, \infty)$. For $\rho > 0$, define
\[
\eta_\rho(x) = \xi(|x|/\rho).
\]
Test \eqref{eq:geodesic} with $\psi_\rho(x) = \eta_\rho(x)x$. This gives
\[
\int_\Omega \left(\eta_\rho(x) + \xi'(|x|/\rho) \frac{|x|}{\rho} \Bigl|\frac{x}{|x|} \cdot \vec{T}\Bigr|^2\right) \, d\|T\|(x) = 0.
\]

Now we compute
\[
\begin{split}
\frac{d}{d\rho} \left(\frac{1}{\rho} \int_\Omega \eta_\rho \, d\|T\|\right) & = -\frac{1}{\rho^2} \int_\Omega \left(\eta_\rho(x) + \frac{|x|}{\rho} \xi'(|x|/\rho)\right) \, d\|T\|(x) \\
& = -\frac{1}{\rho^3} \int_\Omega \xi'(|x|/\rho) |x| \left(1 - \Bigl|\frac{x}{|x|} \cdot \vec{T}\Bigr|^2\right)\, d\|T\|(x). \\
\end{split}
\]
Integrate with respect to $\rho$ over the interval $(s, r)$:
\begin{multline*}
\frac{1}{r} \int_\Omega \eta_r \, d\|T\| - \frac{1}{s} \int_\Omega \eta_s \, d\|T\| \\
\begin{aligned}
& = -\int_s^r \frac{1}{\rho^3} \int_\Omega \xi'(|x|/\rho) |x| \left(1 - \Bigl|\frac{x}{|x|} \cdot \vec{T}\Bigr|^2\right)\, d\|T\|(x) \, d\rho \\
& = -\int_\Omega \int_s^r \frac{|x| \xi'(|x|/\rho)}{\rho^3} \, d\rho \left(1 - \Bigl|\frac{x}{|x|} \cdot \vec{T}\Bigr|^2\right)\, d\|T\|(x).
\end{aligned}
\end{multline*}
Set
\[
\Phi(x) = -\int_s^r \frac{|x| \xi'(|x|/\rho)}{\rho^3} \, d\rho,
\]
so that
\[
\frac{1}{r} \int_\Omega \eta_r \, d\|T\| - \frac{1}{s} \int_\Omega \eta_s \, d\|T\| = \int_\Omega \Phi(x) \left(1 - \Bigl|\frac{x}{|x|} \cdot \vec{T}\Bigr|^2\right)\, d\|T\|(x).
\]
Note that
\[
\Phi(x) = \int_s^r \frac{d}{d\rho} \xi(|x|/\rho) \, \frac{d\rho}{\rho} = \frac{\xi(|x|/r)}{r} - \frac{\xi(|x|/s)}{s} + \int_s^r \xi(|x|/\rho) \, \frac{d\rho}{\rho^2}.
\]
Now let $\xi$ approximate the characteristic function of $(-\infty, 1)$. Then $\Phi(x)$ converges to
\[
r^{-1} + \int_{|x|}^r \frac{d\rho}{\rho^2} = |x|^{-1}
\]
for $s \le |x| < r$ and to $0$ else. Using the dominated convergence theorem, we obtain the formula in the statement.
\end{proof}

\begin{proof}[Proof of Theorem \ref{thm:structure}]
A consequence of Lemma \ref{lem:monotonicity} is that for any $x \in \Omega$, the function $r \mapsto r^{-1}\|T\|(B_r(x))$
is monotone. Therefore, the limit
\[
\Theta(x) = \lim_{r \searrow 0} \frac{\|T\|(B_r(x))}{r}
\]
exists. Lemma \ref{lem:monotonicity} also implies that $\Theta$ is locally bounded.
It follows from standard results on Hausdorff measures
\cite[Section 2.10.19]{Federer:69} that $\|T\|$ is absolutely continuous with respect to
$\Ha^1$.

Let $\Sigma = \{x \in \Omega \colon \Theta(x) > 0\}$. Moreover, for $\ell \in \N$,
let $\Sigma_\ell = \{x \in \Omega \colon \Theta(x) \ge 1/\ell\}$. Then it also
follows from standard results \cite[Section 2.10.19]{Federer:69} that
\[
\lim_{r \searrow 0} \frac{1}{r} \|T\|(B_r(x) \setminus \Sigma_\ell) = 0
\]
for almost every $x \in \Sigma_\ell$ with respect to $\Ha^1$ (and therefore with respect to $\|T\|$ as well). Hence
\[
\lim_{r \searrow 0} \frac{1}{r} \|T\|(B_r(x) \cap \Sigma_\ell) = \Theta(x)
\]
for $\|T\|$-almost every $x \in \Sigma_\ell$. The results of Preiss \cite{Preiss:87} now imply that
the measure $\|T\| \restr \Sigma_\ell$ is countably $1$-rectifiable. Hence $\Sigma$ is a countably
$1$-rectifiable set.

Given $x_0 \in \Sigma$, we now consider tangent measures of $\|T\|$ at $x_0$. To this end, define
$\Omega_r = \frac{1}{r}(\Omega - x_0)$ for $r > 0$. Consider
the measures $\lambda_r$ on $\Omega_r$ defined by
\[
\int_{\Omega_r} \eta \, d\lambda_r = \frac{1}{r} \int_\Omega \eta((x - x_0)/r) \, d\|T\|(x)
\]
for $\eta \in C_0^0(\Omega_r)$, and consider the functions $\vec{T}_r \colon \Omega_r \to \R^{N \times n}$ with
\[
\vec{T}_r(x) = \vec{T}(rx + x_0).
\]

Then for any $R > 0$,
\[
\lambda_r(B_R(0)) = \frac{1}{r} \|T\|(B_{Rr}(x_0)) \to R\Theta(x_0),
\]
while $|\vec{T}_r| = 1$ almost everywhere with respect to $\lambda_r$.
We may therefore pick a sequence $r_\ell \searrow 0$ such that the measure-function pairs
$(\lambda_{r_\ell}, \vec{T}_{r_\ell})$ converge weakly to some measure-function pair
$(\lambda_0, \vec{T}_0)$ over $\R^n$. If $x_0$ is such that $\Sigma$ has an approximate tangent line $L_{x_0} \subseteq \R^n$
at $x_0$ (which is $\|T\|$-almost everywhere \cite[Theorem 11.6]{Simon:83})
and $\vec{T}$ is approximately continuous at $x_0$ with respect to $\|T\|$
(also $\|T\|$-almost everywhere \cite[Theorem 2.9.13]{Federer:69}), then
the limit will be locally strong and of the form
\begin{equation} \label{eq:tangent1}
\lambda_0 = \frac{1}{2} \Theta(x_0) \mathcal{H}^1 \restr L_{x_0}
\end{equation}
and
\begin{equation} \label{eq:tangent2}
\vec{T}_0(0) = \vec{T}(x_0).
\end{equation}

Now suppose that $0 < S < R$. Let $\zeta \in C_0^0(B_R(0) \setminus B_S(0); \R^{N \times n})$ such that
$x \cdot \zeta_k(x) = 0$ everywhere for $k = 1, \dotsc, N$. Then we conclude that
\[
\begin{split}
\lefteqn{\int_{\R^n} |x|^{-1} (\zeta : \vec{T}_0)^2 \, d\lambda_0} \quad \\
& = \lim_{\ell \to \infty} \int_{\R^n} |x|^{-1} (\zeta : \vec{T}_{r_\ell})^2 \, d\lambda_{r_\ell}(x) \\
& = \lim_{\ell \to \infty} \int_{\Omega} |x - x_0|^{-1} \left(\zeta((x - x_0)/r_\ell) : \vec{T}(x)\right)^2 \, d\|T\|(x) \\
& \le \|\zeta\|_{L^\infty(\R^n)}^2 \lim_{\ell \to \infty} \int_{B_{Rr_\ell}(x_0) \setminus B_{Sr_\ell}(x_0)} \frac{1 - \bigl|\frac{x - x_0}{|x - x_0|} \cdot \vec{T}(x)\bigr|^2}{|x - x_0|} \, d\|T\|(x) \\
& = \|\zeta\|_{L^\infty(\R^n)}^2 \lim_{\ell \to \infty} \left(\frac{1}{Rr_\ell} \|T\|(B_{Rr_\ell}(x_0)) - \frac{1}{Sr_\ell} \|T\|(B_{Sr_\ell}(x_0))\right)
\end{split}
\]
by Lemma \ref{lem:monotonicity}. As
\[
\lim_{\ell \to \infty} \frac{\|T\|(B_{Rr_\ell}(x_0))}{Rr_\ell} = \Theta(x_0) = \frac{\|T\|(B_{Sr_\ell}(x_0)}{Sr_\ell},
\]
we conclude that
\[
\int_{\R^n} |x|^{-1} (\zeta : T_0)^2 \, d\lambda_0 = 0
\]
for any $\zeta$ with the above properties. If $x_0$ is such that \eqref{eq:tangent1} and \eqref{eq:tangent2} hold
true, then this means that $\vec{T}_k(x_0) \in L_{x_0}$ for $k = 1, \dotsc, N$.
Recall that this is true for $\|T\|$-almost every $x_0 \in \Sigma$.

Now the claims of the theorem follow almost immediately. We have already seen that
$\|T\|$ is absolutely continuous with respect to $\mathcal{H}^1$. For a set $A$ as in the statement,
we conclude that $\|T\|(A \setminus \Sigma) = 0$ \cite[Section 2.10.19]{Federer:69}, so we may assume that $A \subseteq \Sigma$.
The remaining statements then follow from what we know about $\Sigma$.
\end{proof}

\paragraph{Acknowledgements}
This work was partially supported by the Engineering and Physical Sciences Research Council (grant numbers
EP/X017109/1 and EP/X017206/1).
We wish to thank A.~Backus for his comments on a previous version of this paper.

\def\cprime{$'$}
\providecommand{\bysame}{\leavevmode\hbox to3em{\hrulefill}\thinspace}
\providecommand{\MR}{\relax\ifhmode\unskip\space\fi MR }
\providecommand{\MRhref}[2]{%
  \href{http://www.ams.org/mathscinet-getitem?mr=#1}{#2}
}
\providecommand{\href}[2]{#2}


\begin{thebibliography}{10}

\bibitem{Allard:72}
W.~K. Allard, \emph{On the first variation of a varifold}, Ann. of Math. (2)
  \textbf{95} (1972), 417--491.

\bibitem{Aronsson:65}
G.~Aronsson, \emph{Minimization problems for the functional {${\rm
  sup}_{x}\,F(x,\,f(x),\,f^{\prime} (x))$}}, Ark. Mat. \textbf{6} (1965),
  33--53.

\bibitem{Aronsson:66}
\bysame, \emph{Minimization problems for the functional {${\rm sup}_{x}\, F(x,
  f(x),f\sp\prime (x))$}. {II}}, Ark. Mat. \textbf{6} (1966), 409--431.

\bibitem{Aronsson:67}
\bysame, \emph{Extension of functions satisfying {L}ipschitz conditions}, Ark.
  Mat. \textbf{6} (1967), 551--561.

\bibitem{Aronsson:68}
\bysame, \emph{On the partial differential equation
  {$u_{x}{}^{2}\!u_{xx}+2u_{x}u_{y}u_{xy}+u_{y}{}^{2}\!u_{yy}=0$}}, Ark. Mat.
  \textbf{7} (1968), 395--425 (1968).

\bibitem{Aronsson:84}
\bysame, \emph{On certain singular solutions of the partial differential
  equation {$u^{2}_{x}u_{xx}+2u_{x}u_{y}u_{xy}+u^{2}_{y}u_{yy}=0$}},
  Manuscripta Math. \textbf{47} (1984), 133--151.
  
\bibitem{Backus:24}
A.~Backus, \emph{An $\infty$-Laplacian for differential forms, and calibrated laminations},
  arXiv:2404.02215 [math.AP], 2024.

\bibitem{Bhattacharya-DiBenedetto-Manfredi:89}
T.~Bhattacharya, E.~DiBenedetto, and J.~Manfredi, \emph{Limits as
  {$p\to\infty$} of {$\Delta_pu_p=f$} and related extremal problems}, 1989,
  Some topics in nonlinear PDEs (Turin, 1989), pp.~15--68 (1991).

\bibitem{Brizzi:22}
C.~Brizzi, \emph{On functions with given boundary data and convex constraints
  on the gradient}, arXiv:2209.01462 [math.AP], 2022.

\bibitem{Brizzi-DePascale:23}
C.~Brizzi and L.~De~Pascale, \emph{A property of absolute minimizers in
  {$L^\infty$} calculus of variations and of solutions of the
  {A}ronsson-{E}uler equation}, Adv. Differential Equations \textbf{28} (2023),
  287--310.

\bibitem{Bungert-Korolev:22}
L.~Bungert and Y.~Korolev, \emph{Eigenvalue problems in {$\rm L^\infty$}:
  optimality conditions, duality, and relations with optimal transport}, Comm.
  Amer. Math. Soc. \textbf{2} (2022), 345--373.

\bibitem{Champion-DePascale-Jimenez:09}
T.~Champion, L.~De~Pascale, and C.~Jimenez, \emph{The {$\infty$}-eigenvalue
  problem and a problem of optimal transportation}, Commun. Appl. Anal.
  \textbf{13} (2009), 547--565.

\bibitem{Daskalopoulos-Uhlenbeck:22}
G.~Daskalopoulos and K.~Uhlenbeck, \emph{Transverse measures and best
  {L}ip\-schitz and least gradient maps}, arXiv:2010.06551 [math.DG], 2022.

\bibitem{Dong-Peng-Zhang-Zhou:24}
H.~Dong, F.~Peng, Y.~R.-Y.~Zhang, and Y.~Zhou,
\emph{Jacobian determinants for nonlinear gradient of planar $\infty$-harmonic functions and applications},
J. Reine Angew. Math. \textbf{812} (2024), 59--98.

\bibitem{Evans:03.2}
L.~C. Evans, \emph{Three singular variational problems}, in Viscosity Solutions
  of Differential Equations and Related Topics, vol. 1323, Research Institute
  for the Matematical Sciences, RIMS Kokyuroku, 2003.

\bibitem{Evans-Savin:08}
L.~C. Evans and O.~Savin, \emph{{$C^{1,\alpha}$} regularity for infinity
  harmonic functions in two dimensions}, Calc. Var. Partial Differential
  Equations \textbf{32} (2008), 325--347.

\bibitem{Evans-Smart:11.2}
L.~C. Evans and C.~K. Smart, \emph{Everywhere differentiability of infinity
  harmonic functions}, Calc. Var. Partial Differential Equations \textbf{42}
  (2011), 289--299.

\bibitem{Evans-Yu:05}
L.~C. Evans and Y.~Yu, \emph{Various properties of solutions of the
  infinity-{L}aplacian equation}, Comm. Partial Differential Equations
  \textbf{30} (2005), 1401--1428.

\bibitem{Federer:69}
H.~Federer, \emph{Geometric measure theory}, Springer-Verlag, New York, 1969.

\bibitem{Hutchinson:86}
J.~E. Hutchinson, \emph{Second fundamental form for varifolds and the existence
  of surfaces minimising curvature}, Indiana Univ. Math. J. \textbf{35} (1986),
  45--71.

\bibitem{Jensen:93}
R.~Jensen, \emph{Uniqueness of {L}ipschitz extensions: minimizing the sup norm
  of the gradient}, Arch. Rational Mech. Anal. \textbf{123} (1993), 51--74.

\bibitem{Katzourakis:12}
N.~Katzourakis, \emph{{$L^\infty$} variational problems for maps and the
  {A}ronsson {PDE} system}, J. Differential Equations \textbf{253} (2012),
  2123--2139.

\bibitem{Katzourakis:13}
\bysame, \emph{Explicit {$2D$} {$\infty$}-harmonic maps whose interfaces have
  junctions and corners}, C. R. Math. Acad. Sci. Paris \textbf{351} (2013),
  677--680.

\bibitem{Katzourakis:14.2}
\bysame, \emph{{$\infty$}-minimal submanifolds}, Proc. Amer. Math. Soc.
  \textbf{142} (2014), 2797--2811.

\bibitem{Katzourakis:14.1}
\bysame, \emph{On the structure of {$\infty$}-harmonic maps}, Comm. Partial
  Differential Equations \textbf{39} (2014), 2091--2124.

\bibitem{Katzourakis:15.2}
\bysame, \emph{Nonuniqueness in vector-valued calculus of variations in
  {$L^\infty$} and some linear elliptic systems}, Commun. Pure Appl. Anal.
  \textbf{14} (2015), 313--327.

\bibitem{Katzourakis:17.1}
\bysame, \emph{A characterisation of {$\infty$}-harmonic and {$p$}-harmonic
  maps via affine variations in {$L^\infty$}}, Electron. J. Differential
  Equations (2017), No. 29, 19 pp.

\bibitem{Katzourakis:22}
\bysame, \emph{Generalised vectorial {$\infty$}-eigenvalue nonlinear problems
  for {$L^\infty$} functionals}, Nonlinear Anal. \textbf{219} (2022), Paper No.
  112806, 29.

\bibitem{Katzourakis-Moser:23}
N.~Katzourakis and R.~Moser, \emph{Variational problems in {$L^\infty$}
  involving semilinear second order differential operators}, arXiv:2303.15982
  [math.AP], 2023.
  
\bibitem{Koch-Zhang-Zhou:19}
H.~Koch, Y.~R.-Y.~Zhang and Y.~Zhou, \emph{An asymptotic sharp Sobolev regularity for planar infinity harmonic functions},
J. Math. Pures Appl. (9) \textbf{132} (2019), 457--482.

\bibitem{Kristensen-Mingione:10}
J.~Kristensen and G.~Mingione, \emph{Boundary regularity in variational
  problems}, Arch. Ration. Mech. Anal. \textbf{198} (2010), 369--455.

\bibitem{Mattila:95}
P.~Mattila, \emph{Geometry of sets and measures in {E}uclidean spaces},
  Cambridge Studies in Advanced Mathematics, vol.~44, Cambridge University
  Press, Cambridge, 1995.

\bibitem{Mattila-Moran-Rey:00}
P.~Mattila, M.~Mor\'{a}n, and J.-M. Rey, \emph{Dimension of a measure}, Studia
  Math. \textbf{142} (2000), 219--233.

\bibitem{Preiss:87}
D.~Preiss, \emph{Geometry of measures in {${\bf R}^n$}: distribution,
  rectifiability, and densities}, Ann. of Math. (2) \textbf{125} (1987), no.~3,
  537--643.
  
\bibitem{Savin:05}
O.~Savin, \emph{$C^1$ regularity for infinity harmonic functions in two dimensions},
  Arch. Ration. Mech. Anal. \textbf{176} (2005), no.~3, 351--361.

\bibitem{Sheffield-Smart:12}
S.~Sheffield and C.~K. Smart, \emph{Vector-valued optimal {L}ipschitz
  extensions}, Comm. Pure Appl. Math. \textbf{65} (2012), 128--154.

\bibitem{Simon:83}
L.~Simon, \emph{Lectures on geometric measure theory}, Australian National
  University Centre for Mathematical Analysis, Canberra, 1983.

\bibitem{Thurston:98}
W.~P. Thurston, \emph{Minimal stretch maps between hyperbolic surfaces},
  arXiv:math/9801039 [math.GT], 1998.

\bibitem{Uhlenbeck:77}
K.~Uhlenbeck, \emph{Regularity for a class of non-linear elliptic systems},
  Acta Math. \textbf{138} (1977), 219--240.

\end{thebibliography}
\end{document}